\documentclass[preprint,review,10pt]{elsarticle}
\usepackage{amssymb}
\usepackage{amsfonts}
\usepackage{amsmath}
\usepackage{amsthm}
\usepackage{color}
\usepackage{graphicx}
\usepackage{epsfig,mathrsfs}
\usepackage[bf,SL,BF]{subfigure}
\usepackage{fancyhdr}
\usepackage{CJK}
\usepackage{caption}
\usepackage{wrapfig}
\usepackage{cases}
\usepackage{subfigure}
\graphicspath{{figure/}}
\usepackage{setspace}

\newtheorem{theorem}{Theorem}[section]
\newtheorem{lemma}{Lemma}[section]
\newtheorem*{lemma*}{Lemma A}

\newtheorem{example}{Example}[section]
\newtheorem{remark}{Remark}[section]

\numberwithin{equation}{section}

\renewcommand{\eqref}[1]{(\ref{#1})}

\newcommand{\si}{\fontsize{14pt}{\baselineskip}\selectfont}
 \newcommand{\wuhao}{\fontsize{10pt}{\baselineskip}\selectfont}
\newcommand{\wuhaoa}{\fontsize{8.5pt}{\baselineskip}\selectfont}
 \newcommand{\wu}{\fontsize{11pt}{\baselineskip}\selectfont}
 \newcommand{\sihao}{\fontsize{13pt}{\baselineskip}\selectfont}
\DeclareMathOperator{\Aut}{Aut} \DeclareMathOperator{\Dev}{Dev}
\renewcommand{\baselinestretch}{1.3}
\renewcommand{\figurename}{Fig}
 \allowdisplaybreaks

\begin{document}
 \pagenumbering{arabic}

\begin{frontmatter}
\title{{\bf \noindent Energy  estimates for two-dimensional space-Riesz fractional wave  equation}}
\author{Minghua Chen$^{*}$, Wenshan Yu}
\cortext[cor2]{Corresponding author. E-mail: chenmh@lzu.edu.cn; yuwsh14@lzu.edu.cn}
\address{School of Mathematics and Statistics, Gansu Key Laboratory of Applied Mathematics and Complex Systems, Lanzhou University, Lanzhou 730000, P.R. China }


\begin{abstract}
The fractional wave equation governs the propagation of mechanical diffusive waves in viscoelastic media which exhibits a power-law creep, and consequently provided a
physical interpretation of this equation in the framework of dynamic viscoelasticity.
In this paper,   we first develop the energy method to estimate the one-dimensional space-Riesz fractional wave equation.
For two-dimensional cases with the variable coefficients, the discretized matrices  are proved to be commutative, which ensures to carry out of the priori error estimates.
The unconditional stability and   convergence with the global truncation error $\mathcal{O}(\tau^2+h^2)$ are theoretically
proved and numerically verified.  In particulary,  the framework of the  priori error estimates and  convergence analysis  are still valid for the compact finite difference  scheme and  the nonlocal wave equation.

\medskip
\noindent {\bf Keywords:}
 Riesz fractional wave  equation; Nonlocal wave equation;  Priori error estimates; Energy method; Numerical stability and convergence

\end{abstract}
\end{frontmatter}

\section{Introduction}
The fractional wave equation   is obtained from the classical  wave equation by replacing the  second-order derivative with a fractional derivative of
order $\alpha$, $1<\alpha\leq 2$. Mainardi  \cite{Mainardi:97} pointed out that the fractional wave equation governs
the propagation of mechanical diffusive waves in viscoelastic media which exhibits a power-law creep, and consequently provided a
physical interpretation of this equation in the framework of dynamic viscoelasticity. 
In this paper, we study a second-order accurate numerical method in both space and time  for the two-dimensional space-Riesz fractional wave equation with the variable coefficients 
whose prototype is, for $1< \alpha,\beta\leq 2$,
\begin{equation}\label{1.1}
\frac{\partial^2u(x,y,t)}{\partial t^2}=a(x,y)\frac{\partial^\alpha u(x,y,t)}{\partial |x|^\alpha}+b(x,y)\frac{\partial^\beta u(x,y,t)}{\partial |y|^\beta}+f(x,y,t).
\end{equation}
The initial conditions are
\begin{equation}\label{1.2}
\begin{array}{ll}
&u(x,y,0)=\varphi(x,y) ~~~ {\rm for}~~~ (x,y) \in \Omega,\\
&u_t(x,y,0)=\psi(x,y) ~~~ {\rm for}~~~ (x,y) \in \Omega,
\end{array}
\end{equation}
and the Dirichlet boundary condition
\begin{displaymath}
  u(x,y,t)=0 ~~~ {\rm for}~~~ (x,y) \in \partial \Omega
\end{displaymath}
with  $ \Omega=(0,x_r) \times (0,y_r)$. The function $f(x,y,t)$ is a source
term and all the coefficients   are positive, i.e., $0<a_0\leq a(x,y)\leq a_1$ and $0<b_0\leq b(x,y)\leq b_1$.

The space-Riesz fractional derivative appears in the continuous limit of lattice models with long-range interactions \cite{Tarasov:10},
for $n \in \mathbb{N}$, $n-1 \leq \alpha < n$,
which is defined as  \cite{Podlubny:99}
\begin{equation}\label{1.3}
\frac{\partial^\alpha u(x,y,t)}{\partial |x|^\alpha} =-\kappa_{\alpha}\left( _{0}D_x^{\alpha}+ _{x}\!D_{x_r}^{\alpha} \right)u(x,y,t)  ~~{\rm with}~~\kappa_{\alpha}=\frac{1}{2\cos(\alpha \pi/2)},
\end{equation}
where
\begin{equation*}
\begin{split}
   _{0}D_x^{\alpha}u(x,y,t)
& =\frac{1}{\Gamma(n-\alpha)} \displaystyle \frac{\partial^n}{\partial x^n}
   \int_{0}\nolimits^x{\left(x-\xi\right)^{n-\alpha-1}}{u(\xi,y,t)}d\xi, \\
   _{x}D_{x_r}^{\alpha}u(x,y,t)
&=\frac{(-1)^n}{\Gamma(n-\alpha)}\frac{\partial^n}{\partial x^n}
\int_{x}\nolimits^{x_r}{\left(\xi-x\right)^{n-\alpha-1}}{u(\xi,y,t)}d\xi.
\end{split}
\end{equation*}

For  the Caputo-Riesz time-space fractional  wave equation
${^c}\!{D}_t^\gamma u(x,t)= \frac{\partial^\alpha u(x,t)}{\partial |x|^\alpha} $ with $1<\alpha, \gamma\leq 2$,
Mainardi (2001) et al. obtained the fundamental solution of the space-time fractional diffusion equation \cite{Mainardi:2001}.
Metzler and Nonnenmacher (2002) investigated the physical backgrounds and implications of a space-and time-fractional diffusion and wave equation \cite{{Metzler:02}}.
The numerical solution of space-time fractional diffusion-wave equations are discussed in \cite{Bhrawy:16,Garg:14}, 
but they are lack of the stability and  convergence analysis.
To rewrite the fractional diffusion-wave equation as the  the Volterra type  integro-differential equations,   the stability and  convergence analysis are given with the zero initial conditions \cite{Chen:17}.
For $1<\gamma<2$ and $\alpha=2$, it has been proposed  by various authors \cite{Chen:1992,Cuesta:06,Liu:13,Mclean:1993,Mustapha:13,Sun:06,Yang:14,Zeng:15,Zhang:12}.
For example,  based on the second-order fractional Lubich's methods \cite{Lubich:86},
Cuesta (2006) et al.  derived  the second-order error bounds of the time discretization in a Banach space with the  $\frac{\partial^2 u}{\partial x^2} $ a  sectorial operator  \cite{Cuesta:06}
and Yang (2014) et al. obtained the  second-order convergence  schemes with $1\leq \gamma \leq 1.71832$ \cite{Yang:14}.
For $\gamma=2$ and $1<\alpha<2$,  it seems that achieving a second-order accurate scheme for (\ref{1.1}) is not an easy task
with the nonzero initial conditions.
This paper focuses on providing the weighted numerical scheme to solve the space-Riesz fractional wave equation with the  nonzero initial conditions and the variable coefficients
in one-dimensional and two-dimensional cases. The unconditional stability and   convergence with the global truncation error $\mathcal{O}(\tau^2+h^2)$ are theoretically
proved and numerically verified by the energy method,  which can be easily extended to the nonlocal wave  equation \cite{Du:12}.

The rest of the paper is organized as follows. The next section proposes the second-order accurate scheme for (\ref{1.1}).
In Section 3, we carry out a detailed stability and convergence analysis with the second order accuracy in both time and space directions for the derived schemes.
To show the effectiveness of the schemes, we perform the numerical experiments to verify the theoretical results in Section 4.
The paper is concluded with some remarks in the last section.

\section{Discretization Schemes}\label{sec:1}
Let the mesh points $x_i=ih$, $i=0,1,\ldots,N_x$,
and $t_k=k\tau$, $0\leq k \leq {N_t}$ with $h=x_r/{N_x}$, $\tau=T/{N_t}$,
i.e., $h$ is the uniform space stepsize and $\tau$ the time stepsize. And ${u_i^k}$ denotes the
approximated value of $u(x_i,t_k)$, $a_{i}=a(x_i)$,   $f_{i}^{k}=f(x_i,t_{k})$.

Nowadays, there are already many types of high
order discretization schemes for the Riemann-Liouville space fractional derivatives \cite{Chen:0013,Hao:15,Ji:15,Ortigueira:06,Sousa:12,Tadjeran:06,Tian:12}.
Here, we take the following schemes to approach  (\ref{1.3}), see in \cite{Chen:16,Tian:12}
\begin{equation}\label{2.1}
\begin{split}
_{0}D_{x}^{\alpha}u(x_i)
&=\delta_{x,+}^\alpha u(x_i) +\mathcal{O}(h^2)~~{\rm with}~~\delta_{x,+}^\alpha u(x_i)=\frac{1}{h^{\alpha}}\sum_{m=0}^{i+1}\varphi_{m}^{\alpha}u(x_{i-m+1}),\\
_{x}D_{x_r}^{\alpha}u(x_i)
&=\delta_{x,-}^\alpha u(x_i) +\mathcal{O}(h^2)~~{\rm with}~~ \delta_{x,-}^\alpha u(x_i)=\frac{1}{h^{\alpha}}\sum_{m=0}^{N_x-i+1}\varphi_{m}^{\alpha}u(x_{i+m-1}),
\end{split}
\end{equation}
where
\begin{equation*}
\begin{split}
\varphi_{0}^{\alpha}=\frac{\alpha}{2}g_{0}^{\alpha},~~\varphi_{m}^{\alpha}=\frac{\alpha}{2}g_{m}^{\alpha}+\frac{2-\alpha}{2}g_{m-1}^{\alpha},~~m\geq 1,
\end{split}
\end{equation*}
and
\begin{equation*}
g_m^{\alpha}=(-1)^m\left ( \begin{matrix}\alpha \\ m\end{matrix} \right ),~~{\rm i.e.,}~~
g_0^{\alpha}=1, ~~~~g_m^{\alpha}=\left(1-\frac{\alpha+1}{m}\right)g_{m-1}^{\alpha},~~m \geq 1.
\end{equation*}

Using (\ref{1.3}) and (\ref{2.1}), we obtain the approximation operator of the space-Riesz fractional derivative
\begin{equation}\label{2.2}
\begin{split}
\frac{\partial^\alpha  u(x_i)}{\partial |x|^\alpha}
 =\nabla^\alpha_h u(x_{i}) +\mathcal{O}(h^2)
\end{split}
\end{equation}
with
$$\nabla^\alpha_h u(x_{i})  =-\kappa_{\alpha}\left(\delta_{x,+}^\alpha +\delta_{x,-}^\alpha  \right)u(x_{i})=-\frac{\kappa_{\alpha}}{h^\alpha} \sum_{l=0}^{N_x}\varphi_{i,l}^{\alpha}u(x_{l}),$$
where  $i=1,\ldots,N_x-1$ (together with the zero Dirichlet boundary conditions)  and
\begin{equation*}
\varphi_{i,l}^{\alpha}=\left\{ \begin{array}
 {l@{\quad } l}
  \varphi_{i-l+1}^{\alpha},&l < i-1,\\
  \varphi_{0}^{\alpha}+\varphi_{2}^{\alpha} ,&l=i-1,\\
 2\varphi_{1}^{\alpha},&l=i,\\
\varphi_{0}^{\alpha}+\varphi_{2}^{\alpha} ,&l=i+1,\\
\varphi_{l-i+1}^{\alpha} ,&l>i+1.
\end{array}\right.
\end{equation*}
Taking $u=[u({x_1}),u({x_2}),\cdots,u({x_{N_x-1}})]^{\rm T}$, and  using (\ref{2.1}), (\ref{2.2}), there exists
\begin{equation*}
\begin{split}
& \frac{1}{h^\alpha}\left[\sum_{l=0}^{N_x}\varphi_{1,l}^{\alpha}u(x_{l}),\sum_{l=0}^{N_x}\varphi_{2,l}^{\alpha}u(x_{l}),\ldots,\sum_{l=0}^{N_x}\varphi_{N_x-1 ,l}^{\alpha}u(x_{l}) \right]^T
= \left(\delta_{x,+}^\alpha +\delta_{x,-}^\alpha  \right)u   =  \frac{1}{h^\alpha}A_\alpha u,
\end{split}
\end{equation*}
it yields
\begin{equation}\label{2.3}
\begin{split}
&\nabla^\alpha_h u
=-\kappa_\alpha \left(\delta_{x,+}^\alpha +\delta_{x,-}^\alpha  \right)u   =  \frac{-\kappa_\alpha}{h^\alpha}A_\alpha u,
\end{split}
\end{equation}
where the matrix
\begin{equation}\label{2.4}
A_\alpha=B_{\alpha}+B_{\alpha}^T~~{\rm with} ~~
B_\alpha=\left [ \begin{matrix}
\varphi_1^{\alpha}   &\varphi_2^{\alpha}&\varphi_3^{\alpha}     &      \cdots   &  \varphi_{N_x\!-\!2}^{\alpha}     &  \varphi_{N_x\!-\!1}^{\alpha}  \\
\varphi_0^{\alpha}&  \varphi_1^{\alpha}    &\varphi_2^{\alpha}&  \varphi_3^{\alpha}    &     \cdots   &  \varphi_{N_x\!-\!2}^{\alpha} \\
            &\varphi_0^{\alpha}&\varphi_1^{\alpha}         & \varphi_2^{\alpha}&     \ddots              & \vdots  \\
                   &                  &       \ddots            &        \ddots            &      \ddots             &  \varphi_3^{\alpha}  \\
  &          &              &        \ddots            &   \varphi_1^{\alpha}         &\varphi_2^{\alpha} \\
   &      &          &                   &\varphi_0^{\alpha}& \varphi_1^{\alpha}
 \end{matrix}
 \right ].
\end{equation}

\subsection{Numerical scheme for one-dimensional space-Riesz fractional wave equation}
We now examine the full discretization scheme to the one-dimensional space-Riesz fractional wave equation, i.e,
\begin{equation}\label{2.5}
\frac{\partial^2u(x,t)}{\partial t^2}=a(x)\frac{\partial^\alpha u(x,t)}{\partial |x|^\alpha}+f(x,t)
\end{equation}
with $0<a_0\leq a(x)\leq a_1$  and the zero  Dirichlet boundary condition.  The initial conditions are
\begin{equation}\label{2.6}
\begin{array}{ll}
&u(x,0)=\varphi(x) ~~~ {\rm for}~~~ x \in \Omega,\\
&u_t(x,0)=\psi(x) ~~~ {\rm for}~~~ x \in \Omega.
\end{array}
\end{equation}
In the  time direction derivative, we use the following center difference scheme
\begin{equation}\label{2.7}
  \frac{\partial^2 u(x,t) }{\partial t^2}= \delta^2_t u(x_i,t_k)+\mathcal{O}(\tau^2)
  ~~{\rm with}~~\delta^2_t u(x_i,t_k)= \frac{u_{i}^{k+1}-2u_{i}^k+u_{i}^{k-1}}{\tau^2}.
\end{equation}

In order to achieve an unconditional stable   algorithm, we use the
weighted algorithm for the space-Riesz  fractional derivative, i.e.,
$$\theta u_{i}^{k+1}+(1-2\theta)u_{i}^k+\theta u_{i}^{k-1},~~ \frac{1}{4}\leq \theta\leq 1,$$
to approximate $u(x_i,t_k)$. From (\ref{2.2}) and the above equations,  we can rewrite (\ref{2.5}) as
\begin{equation}\label{2.8}
\begin{split}
&\frac{u(x_i,t_{k+1})-2u(x_i,t_{k})+u(x_i,t_{k-1})}{\tau^2}\\
&  = a(x_i)\nabla^\alpha_h\left[\theta u(x_i,t_{k+1})+(1-2\theta)u(x_i,t_{k})+\theta u(x_i,t_{k-1})\right]
 +f(x_i,t_{k})+ R_i^{k}
\end{split}
\end{equation}
with the local truncation error
\begin{equation}\label{2.9}
 R_i^{k}\leq C_{u,\alpha}(\tau^2+h^2),
\end{equation}
where the constant $C_{u,\alpha}$ is independent of $h$ and $\tau$.
Therefore, the full discretization of (\ref{2.5}) has the following form
\begin{equation}\label{2.10}
  \delta^2_tu^k_i=\theta a_i\nabla^\alpha_hu^{k+1}_i+(1-2\theta)
  a_i\nabla^\alpha_hu^k_i+\theta a_i\nabla^\alpha_hu^{k-1}_i+f_i^k,
\end{equation}
i.e.,
\begin{equation}\label{2.11}
\begin{split}
 &  u_i^{k+1} +\theta \frac{  \tau^2}{h^{\alpha}}{\kappa_{\alpha}a_i} \left[ \sum_{m=0}^{i+1}\varphi_m^{\alpha}u_{i-m+1}^{k+1}
 +\sum_{m=0}^{N_x-i+1}\varphi_m^{\alpha}u_{i+m-1}^{k+1}\right]    \\
 & =2u_i^{k} -(1-2\theta)\frac{  \tau^2}{h^{\alpha}}{\kappa_{\alpha}a_i}\left[ \sum_{m=0}^{i+1}\varphi_m^{\alpha}u_{i-m+1}^{k}
 +\sum_{m=0}^{N_x-i+1}\varphi_m^{\alpha}u_{i+m-1}^{k}\right]    \\
&\quad-u_i^{k-1} -\theta \frac{  \tau^2}{h^{\alpha}}{\kappa_{\alpha}a_i}\left[ \sum_{m=0}^{i+1}\varphi_m^{\alpha}u_{i-m+1}^{k-1}
 +\sum_{m=0}^{N_x-i+1}\varphi_m^{\alpha}u_{i+m-1}^{k-1}\right]  +\tau^2 f_i^{k}.
\end{split}
\end{equation}
Using (\ref{2.5}), (\ref{2.6}) and Taylor expansion with integral form of the remainder, there exists
\begin{equation}\label{2.12}
\begin{split}
u(x_i,\tau)
&=u(x_i,0)+\tau \frac{\partial u(x_i,0)}{\partial t}+\frac{\tau^2}{2} \frac{\partial^2 u(x_i,0)}{\partial t^2}
+\frac{1}{2}\int_0^\tau(\tau-t)^2\frac{\partial^3 u(x_i,t)}{\partial t^3}dt\\
&=\varphi(x_i)+\tau \psi(x_i)+\frac{\tau^2}{2} \left[a(x_i)\frac{\partial^\alpha u(x_i,0)}{\partial |x|^\alpha}+f(x_i,0)\right]
 +\frac{1}{2}\int_0^\tau(\tau-t)^2\frac{\partial^3 u(x_i,t)}{\partial t^3}dt.
\end{split}
\end{equation}
Then we can obtain $u_{i}^1$, i.e.,
\begin{equation}\label{2.13}
\begin{split}
 u_i^1
 &=\varphi(x_i)+\tau\psi(x_i)+\frac{\tau^2}{2}\left[a(x_i)\nabla^\alpha_hu(x_i,0)+f(x_i,0)\right]\\
\end{split}
\end{equation}
with the local truncation error $\mathcal{O}(\tau^3+\tau^2h^2)$, see Section 3.

For the convenience of implementation, we use the matrix form of the grid functions
 \begin{equation*}
 U^{k}=\left[u_1^k,u_2^k,\ldots,u_{N_x-1}^k\right]^{\rm T}, ~~F^{k}=\left[f_1^{k},f_2^{k},\ldots,f_{N_x-1}^{k}\right]^{\rm T}.
  \end{equation*}
Hence, the finite difference scheme (\ref{2.11}) can be recast as
\begin{equation}\label{2.14}
\begin{split}
&\left[I +\theta \frac{  \tau^2}{h^{\alpha}}{\kappa_{\alpha}D}A_{\alpha}  \right] U^{k+1}
= \left[2I-\left(1-2 \theta\right) \frac{  \tau^2}{h^{\alpha}}{\kappa_{\alpha}D}A_{\alpha} \right] U^{k}
-\left[I+\theta\frac{  \tau^2}{h^{\alpha}}{\kappa_{\alpha}D}A_{\alpha} \right] U^{k-1}+\tau F^{k},
\end{split}
\end{equation}
where $A_\alpha$ is defined by (\ref{2.4}) and the diagonal matrix
\begin{equation}\label{2.15}
D=\left [ \begin{matrix}
  a_1                       \\
       & a_{2}                  \\
       &       & \ddots            \\
       &       &      &a_{N_x-1}
 \end{matrix}
 \right ].
\end{equation}

\subsection{Numerical scheme for two-dimensional space-Riesz fractional wave equation}
 Let the mesh points $x_i=ih_x$, $i=0,1,\ldots,N_x$ and $y_j=jh_y$, $j=0,1,\ldots,N_y$ and $t_k=k\tau$, $0\leq k \leq {N_t}$ with  $h_x=x_r/{N_x}$, $h_y=y_r/N_y$, $\tau=T/N_t$. Similarly, we  take $u_{i,j}^k$ as the approximated value of $u(x_i,y_j,t_k)$,
 $a_{i,j}=a(x_i,y_j)$,  $b_{i,j}=b(x_i,y_j)$, $f_{i,j}^{k}=f(x_i,y_j,t_{k})$. We use the center difference
scheme to do the discretization in time direction derivative,
\begin{displaymath}
  \frac{\partial^2 u(x,y,t) }{\partial t^2}= \frac{u_{i,j}^{k+1}-2u_{i,j}^k+u_{i,j}^{k-1}}{\tau^2}+\mathcal{O}(\tau^2),
\end{displaymath}
and the weighted schemes for the space-Riesz fractional derivative, i.e.,
$\theta u_{i,j}^{k+1}+(1-2\theta)u_{i,j}^k+\theta u_{i,j}^{k-1}$ to
approximate $u(x_i,y_j,t_k)$. Therefore  (\ref{1.1}) can be rewritten  as
\begin{equation}\label{2.16}
\begin{split}
&\frac{u(x_i,y_j,t_{k+1})-2u(x_i,y_j,t_{k})+u(x_i,y_j,t_{k-1})}{\tau^2}\\
&  =a(x_i,y_j)\nabla^\alpha_{h_x}\Big(\theta u(x_i,y_j,t_{k+1})+(1-2\theta)u(x_i,y_j,t_{k})+\theta u(x_i,y_j,t_{k-1})\Big)\\
& \quad +  b(x_i,y_j)\nabla^\beta_{h_y}\Big(\theta u(x_i,y_j,t_{k+1})+(1-2\theta)u(x_i,y_j,t_{k})+\theta u(x_i,y_j,t_{k-1})\Big)
 +f(x_i,y_j,t_{k})+ R_{i,j}^{k},
\end{split}
\end{equation}
where the local truncation error is
\begin{equation}\label{2.17}
 R_{i,j}^{k}\leq C_{u,\alpha,\beta}(\tau^2+h_x^2+h_y^2).
\end{equation}
 Similarly,  we  denote
\begin{equation}\label{2.18}
\begin{split}
    &\nabla^\alpha_{h_x}u(x_i,y_j)=-\kappa_{\alpha}\left(\delta_{x,+}^\alpha+ \delta_{x,-}^\alpha\right)u(x_i,y_j)
  ~~ {\rm and}~~ ~\nabla^\beta_{h_y}u(x_i,y_j)=-\kappa_{\beta}\left(\delta_{y,+}^\beta+ \delta_{y,-}^\beta\right)u(x_i,y_j).
\end{split}
\end{equation}
Therefore, the resulting discretization of (\ref{1.1}) has the following form
\begin{equation}\label{2.19}
\begin{split}
 \delta^2_tu^k_{i,j}
 & =\theta a_{i,j}\nabla^\alpha_{h_x}u^{k+1}_{i,j}+(1-2\theta)
  a_{i,j}\nabla^\alpha_{h_x}u^k_{i,j}+\theta a_{i,j}\nabla^\alpha_{h_x}u^{k-1}_{i,j}\\
  &\quad+\theta b_{i,j}\nabla_{h_y}^\beta u^{k+1}_{i,j}+(1-2\theta)
  b_{i,j}\nabla_{h_y}^\beta u^k_{i,j}+\theta b_{i,j}\nabla_{h_y}^\beta u^{k-1}_{i,j}+f_{i,j}^k,
\end{split}
\end{equation}
i.e.,
\begin{equation}\label{2.20}
\begin{split}
&\Big[1-\theta \tau^2 \big(a_{i,j}\nabla^\alpha_{h_x} + b_{i,j}\nabla^\beta_{h_y}\big)\Big]u_{i,j}^{k+1}\\
& =\Big[2+(1-2\theta)  \tau^2 \big(a_{i,j}\nabla^\alpha_{h_x} + b_{i,j}\nabla^\beta_{h_y} \big)\Big]u_{i,j}^{k}
-\Big[1-\theta \tau^2 \big(a_{i,j}\nabla^\alpha_{h_x} + b_{i,j}\nabla^\beta_{h_y} \big)\Big]u_{i,j}^{k-1}
+\tau^2 f_{i,j}^{k}.
\end{split}
\end{equation}
Using (\ref{2.12}) and (\ref{2.13}),  we can obtain
\begin{equation}\label{2.21}
\begin{split}
u_{i,j}^1
&=\varphi(x_i,y_j)+\tau \psi(x_i,y_j)+\frac{\tau^2}{2} \left[\big(  a_{i,j}\nabla^\alpha_{h_x}+b_{i,j}\nabla^\beta_{h_y} \big)u_{i,j}^{0} +f_{i,j}^{0}\right]
\end{split}
\end{equation}
with the local truncation error  $\mathcal{O}(\tau^3+\tau^2h_x^2+\tau^2h_y^2)$, see Section 3.

For the two-dimensional space-Riesz  fractional wave equation (\ref{1.1}), the relevant   perturbation equation of (\ref{2.20}) is of the form
\begin{equation}\label{2.22}
\begin{split}
&\Big(1-\theta\tau^2 a_{i,j}\nabla^\alpha_{h_x} \Big)\Big(1-\theta\tau^2 b_{i,j}\nabla^\beta_{h_y} \Big)u_{i,j}^{k+1}\\
&=\Big[2\Big(1-\theta\tau^2 a_{i,j}\nabla^\alpha_{h_x}  \Big)\Big(1-\theta\tau^2 b_{i,j}\nabla^\beta_{h_y} \Big)
    +\tau^2 a_{i,j}\nabla^\alpha_{h_x} + \tau^2 b_{i,j}\nabla^\beta_{h_y}\Big]        u_{i,j}^{k}\\
 &\quad-\Big(1-\theta\tau^2 a_{i,j}\nabla^\alpha_{h_x}  \Big)\Big(1-\theta\tau^2 b_{i,j}\nabla^\beta_{h_y}\Big)u_{i,j}^{k-1}
+\tau^2 f_{i,j}^{k}.
\end{split}
\end{equation}
Comparing (\ref{2.22}) with (\ref{2.20}), the splitting term is
given by
\begin{equation*}
\theta^2 \tau^4 a_{i,j}b_{i,j}\nabla^\alpha_{h_x}\nabla^\beta_{h_y} \big(u_{i,j}^{k+1}-2u_{i,j}^{k}+u_{i,j}^{k-1}\big),
\end{equation*}
since $\big(u_{i,j}^{k+1}-2u_{i,j}^{k}+u_{i,j}^{k-1}\big)$ is an
$\mathcal{O}(\tau^2)$ term, it implies  that the perturbation
contributes an $\mathcal{O}(\tau^6)$ error component to the truncation error of (\ref{2.20}).
Thus we can rewrite (\ref{1.1}) as
\begin{equation}\label{2.23}
\begin{split}
&\frac{u(x_i,y_j,t_{k+1})-2u(x_i,y_j,t_{k})+u(x_i,y_j,t_{k-1})}{\tau^2} \\
&\quad+\theta^2 \tau^4a_{i,j}b_{i,j}\nabla^\alpha_{h_x}\nabla^\beta_{h_y}\big(u(x_i,y_j,t_{k+1})-2u(x_i,y_j,t_{k})+u(x_i,y_j,t_{k-1})\big)\\
&=a(x_i,y_j)\nabla^\alpha_{h_x}\Big(\theta u(x_i,y_j,t_{k+1})+(1-2\theta)u(x_i,y_j,t_{k})+\theta u(x_i,y_j,t_{k-1})\Big)\\
& \quad +  b(x_i,y_j)\nabla^\beta_{h_y}\Big(\theta u(x_i,y_j,t_{k+1})+(1-2\theta)u(x_i,y_j,t_{k})+\theta u(x_i,y_j,t_{k-1})\Big)
 +f(x_i,y_j,t_{k})+ \widetilde{R}_{i,j}^{k}
\end{split}
\end{equation}
where
\begin{equation}\label{2.24}
\begin{split}
\widetilde{R}_{i,j}^{k}
&={R}_{i,j}^{k}+\theta^2 \tau^4a_{i,j}b_{i,j}\nabla^\alpha_{h_x}\nabla^\beta_{h_y}\big(u(x_i,y_j,t_{k+1})-2u(x_i,y_j,t_{k})+u(x_i,y_j,t_{k-1})\big)\\
&\leq \widetilde{C}_{u,\alpha,\beta}(\tau^2+h_x^2+h_y^2).
\end{split}
\end{equation}
Hence, the system (\ref{2.22}) can be solved by the  alternating direction implicit method (D-ADI)  \cite{Douglas:55,Douglas:64}:
\begin{equation}\label{2.25}
\begin{split}
&\Big(1-\theta \tau^2 a_{i,j}\nabla^\alpha_{h_x}\Big)u_{i,j}^{*}\\
&\quad =2u_{i,j}^k -u_{i,j}^{k-1}
 +\tau^2a_{i,j}\nabla^\alpha_{h_x}\left ((1-2\theta)u_{i,j}^{k}+\theta u_{i,j}^{k-1}\right )
+\tau^2b_{i,j}\nabla^\beta_{h_y}u_{i,j}^k+ \tau^2f_{i,j}^{k},\\
&\Big(1 -\theta \tau^2 b_{i,j}\nabla^\beta_{h_y} \Big)u_{i,j}^{k+1}
=u_{i,j}^{*}+\theta\tau^2 b_{i,j}\nabla^\beta_{h_y}\left (-2 u_{i,j}^{k}+u_{i,j}^{k-1}\right ),
\end{split}
\end{equation}
where $u_{i,j}^{*}$ is an intermediate solution.
Take
\begin{equation*}
\begin{split}
&\mathbf{U}^{k}=[u_{1,1}^k,u_{2,1}^k,\ldots,u_{N_x-1,1}^k,u_{1,2}^k,u_{2,2}^k,\dots,u_{N_x-1,2}^k,\ldots,u_{1,N_y-1}^k,u_{2,N_y-1}^k,\ldots,u_{N_x-1,N_y-1}^k]^T,\\
&\mathbf{F}^{k}=[f_{1,1}^k,f_{2,1}^k,\ldots,f_{N_x-1,1}^k,f_{1,2}^k,f_{2,2}^k,\dots,f_{N_x-1,2}^k,\ldots,f_{1,N_y-1}^k,f_{2,N_y-1}^k,\ldots,f_{N_x-1,N_y-1}^k]^T,\\
\end{split}
\end{equation*}
and denote
\begin{equation}\label{2.26}
\begin{split}
\mathcal{A}_{x} = I \otimes A_{\alpha}
~~{\rm and}~~
\mathcal{A}_{y} =A_{\beta}  \otimes I,
\end{split}
 \end{equation}
where $I$  denotes the unit matrix and  the symbol $\otimes$ the Kronecker product \cite{Laub:05},
and $A_{\alpha}$, $A_{\beta}$ are defined by (\ref{2.4}).
Therefore, we can rewrite  (\ref{2.25}) as the following form
\begin{equation}\label{2.27}
\begin{split}
 \left(I+\theta\frac{  \tau^2}{h_x^{\alpha}}{\kappa_{\alpha}D}\mathcal{A}_x \right)\mathbf{U}^{*}
& = \Big(2I-\left(1-2 \theta\right) \frac{  \tau^2}{h_x^{\alpha}}{\kappa_{\alpha}D}\mathcal{A}_x
    -\frac{  \tau^2}{h_y^{\beta}}{\kappa_{\beta}E}\mathcal{A}_y \Big) \mathbf{U}^{k} \\
&\quad -\left(I+\theta\frac{  \tau^2}{h_x^{\alpha}}{\kappa_{\alpha}D}\mathcal{A}_x \right)\mathbf{U}^{k-1}+\tau^2 \mathbf{F}^{k},\\
 \left(I+\theta\frac{  \tau^2}{h_y^{\beta}}{\kappa_{\beta}E}\mathcal{A}_y\right)\mathbf{U}^{k+1}
 &  = \left(2\theta\frac{ \tau^2}{h_y^{\beta}}{\kappa_{\beta}E}\mathcal{A}_y\right) \mathbf{U}^{k}
  -\left(\theta\frac{ \tau^2}{h_y^{\beta}}{\kappa_{\beta}E}\mathcal{A}_y\right)\mathbf{U}^{k-1}+\mathbf{U}^{*},
\end{split}
\end{equation}
where
 \begin{equation*}
 D=\left [ \begin{matrix}
  D_1                        \\
       &  D_2                 \\
       &       & \ddots            \\
       &       &      & D_{N_y-1}
 \end{matrix}
 \right ]
 ~~{\rm with }~~
 D_j=\left [ \begin{matrix}
  a_{1,j}                        \\
       & a_{2,j}                  \\
       &       & \ddots            \\
       &       &      &a_{N_x-1,j}
 \end{matrix}
 \right ]
\end{equation*}
and
 \begin{equation*}
 E=\left [ \begin{matrix}
  E_1                        \\
       &  E_2                 \\
       &       & \ddots            \\
       &       &      & E_{N_y-1}
 \end{matrix}
 \right ]
 ~~{\rm with }~~
 E_j=\left [ \begin{matrix}
  b_{1,j}                        \\
       & b_{2,j}                  \\
       &       & \ddots            \\
       &       &      &b_{N_x-1,j}
 \end{matrix}
 \right ].
\end{equation*}

\section{Convergence and Stability Analysis}
To rewrite the fractional diffusion-wave equation as the  the Volterra type  integro-differential equations,   the stability and convergence analysis are given with the zero initial conditions \cite{Chen:17}.
Here, we first   develop the energy method to estimate the space-Riesz fractional wave equation with the nonzero initial conditions.
For two-dimensional cases with the variable coefficients, the discretized matrices  are proved to be commutative, which ensures to carry out of the priori error estimates.
\begin{lemma}\cite{Wang:2015}\label{lemma3.1}
Let $\nabla^\alpha_h$ be given in (\ref{2.3}) and $1<\alpha\leq2$. Then there exists an symmetric positive definite  matrix $\Lambda_h^\alpha $ such that
\begin{equation*}
-(\nabla^\alpha_h u,u)>0 \quad and \quad -(\nabla^\alpha_h u,v)=(\Lambda_h^\alpha u , \Lambda_h^\alpha v)
~~{\rm with}~~-\nabla^\alpha_h=\Lambda_h^\alpha\cdot\Lambda_h^\alpha.
\end{equation*}
\end{lemma}
\begin{lemma}(Discrete Gronwall Lemma \cite{Quarteroni:08})\label{lemma3.2}
Assume that $\{a_k\}$ and $\{b_k\}$ is a nonnegative sequence, and the sequence $\varphi^k$ satisfies
$$\varphi^0\leq c_0,~ \varphi^k\leq c_0+\sum_{l=0}^{k-1}b_l+\sum_{l=0}^{k-1}a_l\varphi^l, ~~k\geq 1, $$
where $c_0\geq0.$
Then the sequence $\{\varphi^k\}$ satisfies
$$\varphi^k\leq \left( c_0+\sum_{l=0}^{k-1}b_l\right)\exp\left(\sum_{l=0}^{k-1}a_l\right),~~k\geq 1.$$
\end{lemma}
\begin{lemma}\cite[p.\,141]{Laub:05}\label{lemma3.3}
Let $A \in \mathbb{R}^{n\times n}$ have eigenvalues $\{\lambda_i\}_{i=1}^n$ and $B \in \mathbb{R}^{m\times m}$ have eigenvalues $\{\mu_j\}_{j=1}^m$.
Then the $mn$ eigenvalues of $A \otimes B$ are
\begin{equation*}
  \lambda_1\mu_1,\ldots,\lambda_1\mu_m, \lambda_2\mu_1,\ldots,\lambda_2\mu_m,\ldots,\lambda_n\mu_1\ldots,\lambda_n\mu_m.
\end{equation*}
\end{lemma}
\begin{lemma}\cite[p.\,140]{Laub:05}\label{lemma3.4}
Let $A \in \mathbb{R}^{m\times n}$, $B \in \mathbb{R}^{r\times s}$, $C \in \mathbb{R}^{n\times p}$, and $D \in \mathbb{R}^{s\times t}$.
Then
\begin{equation*}
  (A \otimes B)(C \otimes D)=AC \otimes  BD \quad (\in \mathbb{R}^{mr\times pt}).
\end{equation*}
Moreover, for all $A$ and $B$, $(A \otimes B)^T=A^T\otimes B^T$.
\end{lemma}

\begin{lemma}\label{lemma3.5}
Let $\mathcal{A}_{x} = I \otimes A_{\alpha}$ and $\mathcal{A}_{y} =A_{\beta}  \otimes I$ be defined by (\ref{2.26}). Then
\begin{equation*}
\mathcal{A}_x\mathcal{A}_y=\mathcal{A}_y\mathcal{A}_x,
~~\Lambda_{x}\mathcal{A}_y= \mathcal{A}_y\Lambda_{x}
{\rm~~and~~}
\Lambda_{x}\Lambda_{y}=\Lambda_{y}\Lambda_{x}
{\rm~~with~~}   -A_{\alpha}=\Lambda_{\alpha}\cdot\Lambda_{\alpha},
~~ -A_{\beta}=\Lambda_{\beta}\cdot\Lambda_{\beta}
\end{equation*}
where we denote  $\Lambda_{x}: = I \otimes \Lambda_{\alpha}$
 and $\Lambda_{y}: =\Lambda_{\beta}  \otimes I$.

\end{lemma}
\begin{proof}
From \cite{Wang:2015} or Lemma \ref{lemma3.1}, there exists  $ -A_{\alpha}=\Lambda_{\alpha}\cdot\Lambda_{\alpha}$
and $ -A_{\beta}=\Lambda_{\beta}\cdot\Lambda_{\beta}$, since $-A_\alpha$ and $-A_\beta$ are the symmetric positive definite matrices.
Taking   $\Lambda_{x}: = I \otimes \Lambda_{\alpha}$
 and $\Lambda_{y}: =\Lambda_{\beta}  \otimes I$ and using Lemma \ref{lemma3.4}, the results are obtained.
\end{proof}
\begin{lemma}\label{lemma3.6}
Let $\nabla^\alpha_{h_x}$ and $\nabla^\beta_{h_y}$  be given in (\ref{2.18}) with $1<\alpha,\beta\leq2$.
Then there exist the  symmetric positive definite  matrices $\Lambda_{h_x}^\alpha $ and $\Lambda_{h_y}^\beta $, respectively,  such that
\begin{equation*}
-(\nabla^\alpha_{h_x} \mathbf{U},\mathbf{U})>0 \quad and \quad -(\nabla^\alpha_{h_x} \mathbf{U},\mathbf{V})
=(\Lambda_{h_x}^\alpha \mathbf{U}  , \Lambda_{h_x}^\alpha \mathbf{V})
~~{\rm with}~~-\nabla^\alpha_{h_x}=\Lambda_{h_x}^\alpha\cdot\Lambda_{h_x}^\alpha,
\end{equation*}
and
\begin{equation*}
-(\nabla^\beta_{h_y} \mathbf{U},\mathbf{U})>0 \quad and \quad -(\nabla^\beta_{h_y} \mathbf{U},\mathbf{V})
=(\Lambda_{h_y}^\beta \mathbf{U}  , \Lambda_{h_y}^\beta \mathbf{V})
~~{\rm with}~~-\nabla^\beta_{h_y}=\Lambda_{h_y}^\beta\cdot\Lambda_{h_y}^\beta.
\end{equation*}
\end{lemma}
\begin{proof}
According to (\ref{2.18}) and (\ref{2.26}), it implied that
\begin{equation*}
\begin{split}
&\nabla^\alpha_{h_x} \mathbf{U}
=-\kappa_\alpha \left(\delta_{x,+}^\alpha +\delta_{x,-}^\alpha  \right)\mathbf{U}  =  \frac{-\kappa_\alpha}{h_x^\alpha}\mathcal{A}_x \mathbf{U}.
\end{split}
\end{equation*}
From Lemmas \ref{lemma3.3} and \ref{lemma3.5}, we know  that $\mathcal{A}_{x} = I \otimes A_{\alpha}$ is a symmetric negative definite,
which leads to  $-\nabla^\alpha_{h_x}$ (or $-\nabla^\beta_{h_y}$ ) is the symmetric positive definite.
The proof is completed.
\end{proof}

\subsection{Convergence and stability for one-dimensional space-Riesz fractional wave equation}
First, we introduce some relevant notations and properties of discretized inner product given in \cite{Hu:99,Sun:2005}.
Denote  $u^k=\{u_i^k| 0 \leq i \leq N_x, 0\leq k\leq N_t \}$
and $v^k=\{v_i^k| 0 \leq i \leq N_x, 0\leq k\leq N_t \}$, which are grid functions. And
\begin{equation}\label{3.1}
\begin{split}
&u_{\overline{t},i}^k=(u_{i}^{k} - u_{i}^{k-1})/\tau,
~~~~(u^k,v^k)=h\sum_{i=1}^{N_x-1}u_i^kv_i^k, ~~~~~||u^k||=(u^k,u^k)^{1/2}.
\end{split}
\end{equation}

\begin{lemma}\label{lemma3.7}
Let $\frac{1}{4}\leq \theta \leq 1$, $1< \alpha \leq 2$ and $\{u_{i}^k\}$ be the solution of the  difference scheme
\begin{equation*}
  \delta^2_tu^k_i=\theta a_i\nabla^\alpha_hu^{k+1}_i+(1-2\theta)
  a_i\nabla^\alpha_hu^k_i+\theta a_i\nabla^\alpha_hu^{k-1}_i+f_i^k
\end{equation*}
with the initial conditions and the Dirichlet boundary conditions
\begin{equation*}
\begin{split}
&u_i^0=\varphi_i, ~~0\leq i\leq N_x,\\
&u_i^1=\psi_i,~~0\leq i\leq N_x,\\
&u_0^k=0, ~~u_{N_x}^k=0,~~ 0\leq k\leq N_t-1.
\end{split}
\end{equation*}
 Then
\begin{equation*}
\begin{split}
E^k \leq e^{\frac{3}{2}k\tau}\left[  E^0+\frac{3}{2}\tau\sum_{l=1}^k||f^l||^2\right],
\end{split}
\end{equation*}
where the energy norm is defined by 
\begin{equation*}
\begin{split}
   E^k
   ={\parallel u^{k+1}_{\bar{t}}\parallel}^2+\frac{1}{4}\parallel\sqrt{a}(\Lambda_h^\alpha  u^{k+1}+\Lambda_h^\alpha  u^{k})
   \parallel^2+\frac{1}{4}(4\theta-1)\parallel\sqrt{a}(\Lambda_h^\alpha  u^{k+1}-\Lambda_h^\alpha  u^{k})\parallel^2.
\end{split}
\end{equation*}
\end{lemma}
\begin{proof}
Multiplying   (\ref{2.10}) by  $h( u^{k+1}_i- u^{k-1}_i)$, respectively,  it yields
\begin{equation*}
\begin{split}
   &{\delta^2_t u^k_i\cdot\left[h( u^{k+1}_i- u^k_i)+h( u^k_i- u^{k-1}_i)\right]}
   =h( u_{\bar{t},i}^{k+1})^2-h( u_{\bar{t},i}^k)^2,
\end{split}
\end{equation*}
  and
\begin{equation*}
  \left[\theta a_i\nabla^\alpha_h u^{k+1}_i+(1-2\theta)a_i
  \nabla^\alpha_h u^k_i+\theta a_i\nabla^\alpha_h u^{k-1}_i+f_i^k\right]\cdot h( u_i^{k+1}- u_i^{k-1}).
\end{equation*}
Then summing up for $i$ from 1 to $N_x-1$ for the  above equations, respectively, there exists
\begin{equation}\label{3.2}
   \sum_{i=1}^{N_x-1}\left[h( u_{\bar{t},i}^{k+1})^2-h( u_{\bar{t},i}^k)^2\right]
   ={\parallel  u^{k+1}_{\bar{t}}\parallel}^2- {\parallel  u^{k}_{\bar{t}}\parallel}^2,
\end{equation}
and
\begin{equation}\label{3.3}
\begin{split}
  &\sum_{i=1}^{N_x-1}{\left[\theta a_i\nabla^\alpha_h u^{k+1}_i
  +(1-2\theta)a_i\nabla^\alpha_h u^k_i+\theta a_i\nabla^\alpha_h u^{k-1}_i+f_i^k\right]\cdot h( u_i^{k+1}- u_i^{k-1})}\\
  &=I_1+I_2+(f^{k},u^{k+1}-u^{k-1}),
\end{split}
\end{equation}
where
\begin{equation*}
\begin{split}
   I_1
   = \theta(a\nabla^\alpha_h  u^{k+1}+a\nabla^\alpha_h  u^{k-1}, u^{k+1}- u^{k-1}),
  ~~ I_2
   = (1-2\theta)(a\nabla^\alpha_h  u^{k}, u^{k+1}- u^{k-1}).
\end{split}
\end{equation*}
According to Lemma  \ref{lemma3.1}, which leads to
\begin{equation*}
\begin{split}
  I_1
  &=-\theta\left[a\Lambda^\alpha _h( u^{k+1}+ u^{k-1}),\Lambda^\alpha _h( u^{k+1}- u^{k-1})\right]
  =-\theta\left(\parallel\sqrt{a}\Lambda^\alpha _h u^{k+1}\parallel^2-\parallel\sqrt{a}\Lambda^\alpha _h u^{k-1}\parallel^2\right),
\end{split}
\end{equation*}
and
\begin{equation*}
\begin{split}
  I_2
  &=-(1-2\theta)\left[(a\Lambda^\alpha _h u^k,\Lambda^\alpha _h u^{k+1})-(a\Lambda^\alpha _h u^k,\Lambda^\alpha _h u^{k-1})\right]\\
  &=-\frac{(1-2\theta)}{4}\Big[(a\Lambda^\alpha _h u^k+a\Lambda^\alpha _h u^{k+1},\Lambda^\alpha _h u^k+\Lambda^\alpha _h u^{k+1})
  -(a\Lambda^\alpha _h u^k-a\Lambda^\alpha _h u^{k+1},\Lambda^\alpha _h u^k-\Lambda^\alpha _h u^{k+1})\\
   &\quad-(a\Lambda^\alpha _h u^k+a\Lambda^\alpha _h u^{k-1},\Lambda^\alpha _h u^k+\Lambda^\alpha _h u^{k-1})
    +(a\Lambda^\alpha _h u^k-a\Lambda^\alpha _h u^{k-1},\Lambda^\alpha _h u^k-\Lambda^\alpha _h u^{k-1})\Big]\\
  &=-\frac{(1-2\theta)}{4}\Big(\parallel\sqrt{a}(\Lambda^\alpha _h u ^{k+1}
    +\Lambda^\alpha _h u ^k)\parallel^2-\parallel\sqrt{a}(\Lambda^\alpha _h u ^{k+1}-\Lambda^\alpha _h u ^{k})\parallel^2\\
   &\quad-\parallel\sqrt{a}(\Lambda^\alpha _h u ^k+\Lambda^\alpha _h u ^{k-1})\parallel^2+\parallel\sqrt{a}(\Lambda^\alpha _h u
  ^k-\Lambda^\alpha _h u ^{k-1})\parallel^2\Big).
\end{split}
\end{equation*}
Combine (\ref{2.10}), (\ref{3.2}) and (\ref{3.3}), we obtain
\begin{equation}\label{3.4}
     {\parallel  u^{k+1}_{\bar{t}}\parallel}^2- {\parallel  u^{k}_{\bar{t}}\parallel}^2-I_1-I_2=(f^{k},u^{k+1}-u^{k-1}),
\end{equation}
i.e.,
\begin{equation*}
\begin{split}
  &{\parallel  u^{k+1}_{\bar{t}}\parallel}^2+\theta\parallel\sqrt{a}\Lambda_h^\alpha  u^{k+1}\parallel^2
  +\frac{1-2\theta}{4}\left(\parallel\sqrt{a}(\Lambda_h^\alpha  u^{k+1}+\Lambda_h^\alpha  u^{k})\parallel^2
  -\parallel\sqrt{a}(\Lambda_h^\alpha  u^{k+1}-\Lambda_h^\alpha  u^{k})\parallel^2\right)\\
  &={\parallel  u^{k}_{\bar{t}}\parallel}^2+\theta\parallel\sqrt{a}\Lambda_h^\alpha  u^{k-1}\parallel^2
  +\frac{1-2\theta}{4}\left(\parallel\sqrt{a}(\Lambda_h^\alpha  u^{k}+\Lambda_h^\alpha  u^{k-1})\parallel^2
  -\parallel\sqrt{a}(\Lambda_h^\alpha  u^{k}-\Lambda_h^\alpha  u^{k-1})\parallel^2\right)\\
  &\quad+(f^{k},u^{k+1}-u^{k-1}).
\end{split}
\end{equation*}
Adding $\theta\parallel\sqrt{a}\Lambda_h^\alpha  u^{k}\parallel^2$ on both sides of  the above equation, there exists
\begin{equation*}
\begin{split}
  &{\parallel  u ^{k+1}_{\bar{t}}\parallel}^2+\theta(\parallel\sqrt{a}\Lambda_h^\alpha  u ^{k+1}\parallel^2+\parallel\sqrt{a}\Lambda_h^\alpha u ^{k}\parallel^2)\\
  &\quad+\frac{1-2\theta}{4}(\parallel\sqrt{a}(\Lambda_h^\alpha  u^{k+1}+\Lambda_h^\alpha  u^{k})\parallel^2
  -\parallel\sqrt{a}(\Lambda_h^\alpha  u^{k+1}-\Lambda_h^\alpha  u^{k})\parallel^2)\\
  &={\parallel u ^{k}_{\bar{t}}\parallel}^2+\theta(\parallel\sqrt{a}\Lambda_h^\alpha  u ^{k}\parallel^2+\parallel\sqrt{a}\Lambda_h^\alpha u ^{k-1}\parallel^2)\\
  &\quad+\frac{1-2\theta}{4}(\parallel\sqrt{a}(\Lambda_h^\alpha  u^{k}+\Lambda_h^\alpha  u^{k-1})\parallel^2
  -\parallel\sqrt{a}(\Lambda_h^\alpha  u^{k}-\Lambda_h^\alpha  u^{k-1})\parallel^2)
  +(f^{k},u^{k+1}-u^{k-1}).
\end{split}
\end{equation*}
Denoting
\begin{equation*}
\begin{split}
   E^k
   &={\parallel u^{k+1}_{\bar{t}}\parallel}^2+\theta(\parallel\sqrt{a}\Lambda_h^{\frac{\alpha}
   {2}} u^{k+1}\parallel^2+\parallel\sqrt{a}\Lambda_h^\alpha  u^{k}\parallel^2)\\
  &\quad+\frac{1-2\theta}{4}\Big(\parallel\sqrt{a}(\Lambda_h^\alpha  u^{k+1}+\Lambda_h^\alpha u^{k})\parallel^2
  -\parallel\sqrt{a}(\Lambda_h^\alpha  u^{k+1}-\Lambda_h^\alpha  u^{k})\parallel^2\Big),
\end{split}
\end{equation*}
i.e.,
\begin{equation}\label{3.5}
\begin{split}
   E^k
   ={\parallel u^{k+1}_{\bar{t}}\parallel}^2+\frac{1}{4}\parallel\sqrt{a}(\Lambda_h^\alpha  u^{k+1}+\Lambda_h^\alpha  u^{k})
   \parallel^2+\frac{1}{4}(4\theta-1)\parallel\sqrt{a}(\Lambda_h^\alpha  u^{k+1}-\Lambda_h^\alpha  u^{k})\parallel^2,
\end{split}
\end{equation}
where we use
 $$\parallel\sqrt{a}\Lambda_h^\alpha  u^{k}\parallel^2
   +\parallel\sqrt{a}\Lambda_h^\alpha  u^{k-1}\parallel^2
   =\frac{1}{2}\left(\parallel\sqrt{a}(\Lambda_h^\alpha  u^{k}+\Lambda_h^\alpha  u^{k-1})
   \parallel^2+\parallel\sqrt{a}(\Lambda_h^\alpha  u^{k}-\Lambda_h^\alpha  u^{k-1})\parallel^2\right).$$
From
\begin{equation}\label{3.6}
\begin{split}
  (f^{k}, u^{k+1}- u^{k-1})
  &=h\tau\sum_{i=1}^{N_x-1}2f_i^k\left(\frac{ u_i^{k+1}- u_i^{k-1}}{2\tau}\right)\\
  &\leq {h\tau}\sum_{i=1}^{N_x-1}\left[\left(f_i^k\right)^2+\left(\frac{ u_i^{k+1}- u_i^{k}+ u_i^{k}- u_i^{k-1}}{2\tau}\right)^2\right]\\
  &\leq \frac{\tau}{2}\left({\parallel  u^{k+1}_{\bar{t}}\parallel}^2 +{\parallel  u^{k}_{\bar{t}}\parallel}^2\right)+\tau||f^k||^2,
\end{split}
\end{equation}
and (\ref{3.5}),  (\ref{3.4}),  we obtain
\begin{equation*}
\begin{split}
  E^{k}-E^{k-1}=(f^{k}, u^{k+1}- u^{k-1}) \leq \frac{\tau}{2} (E^{k}+E^{k-1}) +\tau||f^k||^2,
\end{split}
\end{equation*}
i.e,
\begin{equation*}
\begin{split}
  \left(1-\frac{\tau}{2}\right)E^{k}\leq \left(1+\frac{\tau}{2}\right)E^{k-1}+\tau||f^k||^2.
\end{split}
\end{equation*}
Therefore, for $\tau\leq 2/3$, it yields
\begin{equation*}
\begin{split}
  E^{k}\leq \left(1+\frac{3\tau}{2}\right)E^{k-1}+\frac{3}{2}\tau||f^k||^2.
\end{split}
\end{equation*}
Using the discrete Gronwall inequality (see Lemma \ref{lemma3.2}), we have
\begin{equation*}
\begin{split}
  E^k \leq e^{\frac{3}{2}k\tau}\left[  E^0+\frac{3}{2}\tau\sum_{l=1}^k||f^l||^2\right].
\end{split}
\end{equation*}
The proof is completed.
\end{proof}
\begin{theorem}\label{theorem:3.1}
Let $u(x_i,t_k)$ be the exact solution of (\ref{2.5}) with $1<\alpha\leq2$, $\frac{1}{4}\leq\theta\leq1$;
$u^k_i$ be the solution of the finite difference scheme (\ref{2.10}) and $e_i^k=u(x_i,t_k)-u^k_i$.  Then
\begin{equation*}
\begin{split}
  E^k=\mathcal{O}(\tau^2+h^2)^2,
\end{split}
\end{equation*}
where  the energy norm is defined by 
\begin{equation*}
\begin{split}
   E^k
   ={\parallel e^{k+1}_{\bar{t}}\parallel}^2+\frac{1}{4}\parallel\sqrt{a}(\Lambda_h^\alpha  e^{k+1}+\Lambda_h^\alpha  e^{k})
   \parallel^2+\frac{1}{4}(4\theta-1)\parallel\sqrt{a}(\Lambda_h^\alpha  e^{k+1}-\Lambda_h^\alpha  e^{k})\parallel^2.
\end{split}
\end{equation*}
\end{theorem}

\begin{proof}
Subtracting  (\ref{2.10}) from  (\ref{2.8}),  it yields
\begin{equation}\label{3.7}
   \delta^2_te^k_i=\theta a_i\nabla^\alpha_he^{k+1}_i+(1-2\theta)a_i
   \nabla^\alpha_he^k_i+\theta a_i\nabla^\alpha_he^{k-1}_i+R^{k}_i.
\end{equation}
Using Lemma \ref{lemma3.7}, we obtain
\begin{equation}\label{3.8}
\begin{split}
  E^k \leq e^{\frac{3}{2}k\tau}\left[  E^0+\frac{3}{2}\tau\sum_{l=1}^k||R^l||^2\right],
\end{split}
\end{equation}
where
\begin{equation}\label{3.9}
\begin{split}
   E^k
   ={\parallel e^{k+1}_{\bar{t}}\parallel}^2+\frac{1}{4}\parallel\sqrt{a}(\Lambda_h^\alpha  e^{k+1}+\Lambda_h^\alpha  e^{k})
   \parallel^2+\frac{1}{4}(4\theta-1)\parallel\sqrt{a}(\Lambda_h^\alpha  e^{k+1}-\Lambda_h^\alpha  e^{k})\parallel^2.
\end{split}
\end{equation}
Next we estimate the local error truncation of $E^0$. Since $e_i^0=0$ and
\begin{equation*}
 \begin{split}
    e_i^1
  &=\frac{\tau^2}{2} \left[a(x_i)\left(\frac{\partial^\alpha u(x_i,0)}{\partial |x|^\alpha}-\nabla^\alpha_h\varphi(x_i)\right)\right]
  +\frac{1}{2}\int_0^\tau(\tau-t)^2\frac{\partial^3 u(x_i,t)}{\partial t^3}dt\\
  &=\frac{\tau^2}{2}a(x_i)C_{1,\alpha}\frac{\partial^{\alpha+2}u(\xi_i,t)}{\partial |x|^{\alpha+2}}h^2+\frac{1}{2}\int_0^\tau{(\tau-t)^2\frac{\partial^3u(x_i,t)}{\partial t^3}dt}
  \leq C_{2,\alpha}(\tau^3+\tau^2h^2),
\end{split}
\end{equation*}
where $\xi_i \in (0,x_r)$ and
\begin{equation*}
\begin{split}
C_{2,\alpha}= \max\limits_{0\leq x\leq x_r,0\leq t\leq T}\Big\{\frac{1}{2}a_1\Big|{C_{1,\alpha}\frac{\partial^{\alpha+2}u(\xi_i,t)}{\partial x^{\alpha+2}}}\Big|,\frac{1}{6}{\int_0^\tau\Big|\frac{\partial^3u(x_i,t)}{\partial t^3}\Big|dt}\Big\},
\end{split}
\end{equation*}
it implies  that
\begin{equation}\label{3.10}
\begin{split}
   \parallel e^1_{\overline{t}}\parallel^2
   &=\parallel \frac{e^1-e^0}{\tau}\parallel^2\leq(N_x-1)h\frac{1}{\tau^2}C_{2,\alpha}(\tau^3+\tau^2h^2)\cdot C_{2,\alpha}(\tau^3+\tau^2h^2)
   \leq C_{2,\alpha}^2x_r(\tau^2+\tau h^2)^2.
\end{split}
\end{equation}
Here, the coefficients  $ C_{l,\alpha}, 1\leq l\leq2$ are the  constants independent of $h$ and $\tau$.

According to (\ref{2.2}) and the above equations, there exists
\begin{equation*}
\begin{split}
  &\parallel\sqrt{a}\Lambda_h^\alpha e^{1}\parallel^2
  =-(a\nabla_h^\alpha e^{1},e^{1})=-h\sum_{i=1}^{N_x-1}a_{i}\left(\nabla_h^\alpha e^{1}_i\right)\cdot e_i^1\\
  &=-h\sum_{i=1}^{N_x-1} a_{i} \sum_{l=0}^{N_x}\frac{-\kappa_{\alpha}}{h^\alpha}\varphi_{i,l}^{\alpha}
  \left[\frac{a_iC_{1,\alpha}}{2}\frac{\partial^{\alpha+2}u(\xi_i,t)}{\partial |x|^{\alpha+2}}\tau^2h^2
  +\frac{1}{2}\int_0^\tau{(\tau-t)^2\frac{\partial^3u(x_i,t)}{\partial t^3}dt}\right]\cdot e_i^1\\
  &=-h\sum_{i=1}^{N_x-1}a_{i}
  \left[\frac{a_iC_{1,\alpha}}{2}\frac{\partial^{2\alpha+2}u(\xi_i,t)}{\partial |x|^{2\alpha+2}}\tau^2h^2
  +C_{3,\alpha}\frac{\partial^{2\alpha+4}u(\overline{\xi}_i,t)}{\partial |x|^{2\alpha+4}}\tau^2h^4\right]\cdot e_i^1\\
  &\quad-h\sum_{i=1}^{N_x-1}\frac{a_{i}}{2}\int_0^\tau{(\tau-t)^2
  \left[\frac{\partial^{\alpha+3}u(x_i,t)}{\partial t^3\partial |x|^\alpha}
  +C_{4,\alpha}\frac{\partial^{\alpha+5}u(\widetilde{\xi}_i,t)}{\partial t^3 \partial |x|^{\alpha+2}}h^2\right]dt}\cdot e_i^1\\
  &   \leq  C_{5,\alpha}(\tau^3+\tau^2h^2)\cdot C_{2,\alpha}(\tau^3+\tau^2h^2),
\end{split}
\end{equation*}
where $\xi_i,\overline{\xi}_i, \widetilde{\xi_i} \in (0,x_r)$ and $ C_{l,\alpha}, 1\leq l\leq 5$ are the  constants independent of $h$ and $\tau$.
Using  (\ref{3.9}), (\ref{3.10}) and the above equation, we have
\begin{equation}\label{3.11}
  E^0\leq  C_\alpha^2x_r(\tau^2+\tau h^2)^2
\end{equation}
with a constant  $C_\alpha$.
From (\ref{2.9}), (\ref{3.8}) and  (\ref{3.11}), it means that
\begin{equation*}
\begin{split}
  E^k
  &\leq e^{\frac{3}{2}k\tau}\left[  C_\alpha^2x_r(\tau^2+\tau h^2)^2+\frac{3}{2}k\tau C^2_{u,\alpha}(\tau^2+h^2)^2\right]
  \leq  \widetilde{C}_\alpha e^{\frac{3}{2}T}(\tau^2+h^2)^2
\end{split}
\end{equation*}
with $ \widetilde{C}_\alpha=2\max\{ C_\alpha^2x_r, \frac{3}{2}C^2_{u,\alpha}T\}$. The proof is completed.
\end{proof}
\begin{theorem}\label{theorem3.2}
The difference scheme (\ref{2.14}) with $1<\alpha\leq2$ and $\frac{1}{4}\leq\theta\leq1$ is unconditionally stable.
\end{theorem}
\begin{proof}
  From Lemma \ref{lemma3.7}, the proof is completed.
\end{proof}
\subsection{Convergence and stability for two-dimensional space-Riesz fractional wave  equation}
Denote  $u^k=\{u_{i,j}^k| 0 \leq i \leq N_x,0 \leq j \leq N_y, 0\leq k\leq N_t \}$
and $v^k=\{v_{i,j}^k| 0 \leq i \leq N_x,0 \leq j \leq N_y, 0\leq k\leq N_t \}$, which are grid functions. And
\begin{equation}\label{3.12}
\begin{split}
&u_{\overline{t},i,j}^k=(u_{i,j}^{k} - u_{i,j}^{k-1})/\tau,
~~~~(u^k,v^k)=h_xh_y\sum_{i=1}^{N_x-1}\sum_{j=1}^{N_y-1}u_{i,j}^kv_{i,j}^k, ~~~~~||u^k||=(u^k,u^k)^{1/2}.
\end{split}
\end{equation}
\begin{lemma}\label{lemma3.8}
Let $\frac{1}{4}\leq \theta \leq 1$, $1< \alpha,\beta \leq 2$ and $\{u_{ij}^k\}$ be the solution of the difference scheme
\begin{equation}\label{3.13}
\begin{split}
  &\delta^2_tu^k_{i,j}+\theta^2 \tau^4 a_{i,j}b_{i,j}\nabla^\alpha_{h_x}\nabla^\beta_{h_y} \big(u_{i,j}^{k+1}-2u_{i,j}^{k}+u_{i,j}^{k-1}\big)\\
  &=\theta a_{i,j}\nabla^\alpha_{h_x}u^{k+1}_{i,j}+(1-2\theta)
  a_{i,j}\nabla^\alpha_{h_x}u^k_{i,j}+\theta a_{i,j}\nabla^\alpha_{h_x}u^{k-1}_{i,j}\\
  &\quad+\theta b_{i,j}\nabla_{h_y}^\beta u^{k+1}_{i,j}+(1-2\theta)
  b_{i,j}\nabla_{h_y}^\beta u^k_{i,j}+\theta b_{i,j}\nabla_{h_y}^\beta u^{k-1}_{i,j}+f_{i,j}^k
  \end{split}
\end{equation}
with the initial conditions and the  Dirichlet boundary conditions
\begin{equation*}
\begin{split}
&u_{i,j}^0=\varphi_{i,j}, ~~0\leq i\leq N_x,0\leq j\leq N_y,\\
&u_{i,j}^1=\psi_{i,j},~~0\leq i\leq N_x,0\leq j\leq N_y,\\
& u_{i,j}^k=0, ~~~ (x_i,y_j) \in \partial \Omega,~~~0\leq k\leq N_t-1.
\end{split}
\end{equation*}
 Then
\begin{equation*}
\begin{split}
E^k \leq e^{\frac{3}{2}k\tau}\left[  E^0+\frac{3}{2}\tau\sum_{l=1}^k||f^l||^2\right],
\end{split}
\end{equation*}
where the energy  norm is defined by 
\begin{equation*}
\begin{split}
   E^k
   &={\parallel u^{k+1}_{\bar{t}}\parallel}^2+\frac{1}{4}\parallel\sqrt{a}(\Lambda_{h_x}^\alpha  u^{k+1}+\Lambda_{h_x}^\alpha  u^{k})
   \parallel^2+\frac{1}{4}(4\theta-1)\parallel\sqrt{a}(\Lambda_{h_x}^\alpha  u^{k+1}-\Lambda_{h_x}^\alpha  u^{k})\parallel^2\\
   &\quad+\frac{1}{4}\parallel\sqrt{b}(\Lambda_{h_y}^\beta  u^{k+1}+\Lambda_{h_y}^\beta  u^{k})
   \parallel^2+\frac{1}{4}(4\theta-1)\parallel\sqrt{b}(\Lambda_{h_y}^\beta  u^{k+1}-\Lambda_{h_y}^\beta  u^{k})\parallel^2\\
   &\quad+\theta^2\tau^6\parallel\sqrt{ab}\Lambda_{h_x}^\alpha\Lambda_{h_y}^\beta u^{k+1}_{\overline{t}}\parallel^2.
\end{split}
\end{equation*}
\end{lemma}
\begin{proof}
Multiplying  (\ref{3.13}) by $h_xh_y( u ^{k+1}_{i,j}- u ^{k-1}_{i,j})$ and using  Lemmas \ref{lemma3.5}, \ref{lemma3.6},  there exists
  \begin{equation*}
  \begin{split}
&\left(\delta^2_t u^{k}_{i,j}+\theta^2\tau^4a_{i,j}b_{i,j}\nabla_{h_x}^\alpha\nabla_{h_y}^\beta (u_{i,j}^{k+1}-2u_{i,j}^k+u_{i,j}^{k-1})\right)
    \cdot \left[h_xh_y( u^{k+1}_{i,j}- u^{k}_{i,j})+h_xh_y( u^{k}_{i,j}- u^{k-1}_{i,j})\right]\\
 &\!=\!h_xh_y( u_{\bar{t},i,j}^{k+1})^2
   \!-\!h_xh_y( u_{\bar{t},i,j}^k)^2
   \!+\!h_xh_y\theta^2\tau^6(\sqrt{a_{i,j}b_{i,j}}\Lambda_{h_x}^\alpha\Lambda_{h_y}^\beta u^{k+1}_{\overline{t},i,j})^2
  \!-\!h_xh_y\theta^2\tau^6(\sqrt{a_{i,j}b_{i,j}}\Lambda_{h_x}^\alpha\Lambda_{h_y}^\beta u^{k}_{\overline{t},i,j})^2,\\
  \end{split}
  \end{equation*}
  and
\begin{equation*}
\begin{split}
&\Big[\theta a_{i,j}\nabla^\alpha_{h_x} u^{k+1}_{i,j}+(1-2\theta)a_{i,j}
  \nabla^\alpha_{h_x} u^{k}_{i,j}+\theta a_{i,j}\nabla^\alpha_{h_x} u^{k-1}_{i,j}\\
  &\quad+\theta  u^{k+1}_{i,j}+(1-2\theta)b_{i,j}
  \nabla_{h_y}^\beta  u^{k}_{i,j}+\theta b_{i,j}\nabla_{h_y}^\beta  u^{k-1}_{i,j}+f_{i,j}^k\Big]\cdot h_xh_y( u_{i,j}^{k+1}- u_{i,j}^{k-1}).
\end{split}
\end{equation*}
Then summing up for $i$ from 1 to $N_x-1$ and for $j$ from 1 to $N_y-1$, we have
\begin{equation}\label{3.14}
\begin{split}
&\sum_{i=1}^{N_x-1}\sum_{j=1}^{N_y-1}\left[h_xh_y( u_{\bar{t},i,j}^{k+1})^2-h_xh_y( u_{\bar{t},i,j}^k)^2\right]
 ={\parallel  u^{k+1}_{\bar{t}}\parallel}^2- {\parallel  u^{k}_{\bar{t}}\parallel}^2,\\
    &\sum_{i=1}^{N_x-1}\sum_{j=1}^{N_y-1}\Big(h_xh_y\theta^2\tau^6(\sqrt{a_{i,j}b_{i,j}}\Lambda_{h_x}^\alpha\Lambda_{h_y}^\beta u^{k+1}_{\overline{t},i,j})^2 -h_xh_y\theta^2\tau^6(\sqrt{a_{i,j}b_{i,j}}\Lambda_{h_x}^\alpha\Lambda_{h_y}^\beta u^{k}_{\overline{t},i,j})^2\Big)\\
    &=\theta^2\tau^6\parallel\sqrt{ab}\Lambda_{h_x}^\alpha\Lambda_{h_y}^\beta u^{k+1}_{\overline{t}}\parallel^2-\theta^2\tau^6\parallel\sqrt{ab}\Lambda_{h_x}^\alpha\Lambda_{h_y}^\beta u^{k}_{\overline{t}})\parallel^2,
\end{split}
\end{equation}
 and
\begin{equation}\label{3.15}
\begin{split}
  &\sum_{i=1}^{N_x-1}\sum_{j=1}^{N_y-1}\Big[\theta a_{i,j}\nabla^\alpha_{h_x} u^{k+1}_{i,j}+(1-2\theta)a_{i,j}
  \nabla^\alpha_{h_x} u^{k}_{i,j}+\theta a_{i,j}\nabla^\alpha_{h_x} u^{k-1}_{i,j} \\
  &\quad+\theta b_{i,j}\Lambda_{h_y}^\beta  u ^{k+1}_{i,j}+(1-2\theta)b_{i,j}\Lambda_{h_y}^\beta  u ^{k}_{i,j}+\theta b_{i,j}\Lambda_{h_y}^\beta u ^{k-1}_{i,j}+f_{i,j}^k\Big]\cdot h_xh_y( u _{i,j}^{k+1}- u _{i,j}^{k-1})\\
   &=I_1+I_2+I_3+I_4+(f^k, u^{k+1}- u^{k-1}),
\end{split}
\end{equation}
where
\begin{equation*}
\begin{split}
I_1
& = \theta\left(a\nabla^\alpha_{h_x}  u^{k+1}+a\nabla^\alpha_{h_x}  u^{k-1}, u^{k+1}- u^{k-1}\right),
~I_2
= (1-2\theta)\left(a\nabla^\alpha_{h_x}  u^{k}, u^{k+1}- u^{k-1}\right),\\
I_3
&= \theta\left(b\nabla_{h_y}^\beta   u^{k+1}+b\nabla_{h_y}^\beta   u^{k-1}, u^{k+1}- u^{k-1}\right),
~I_4
= (1-2\theta)\left(b\nabla_{h_y}^\beta   u^{k}, u^{k+1}- u^{k-1}\right).
\end{split}
\end{equation*}
  According to Lemma \ref{lemma3.6}, we have
  \begin{equation*}
  \begin{split}
  I_1
  =&-\theta\left(\parallel\sqrt{a}\Lambda^\alpha_{h_x} u^{k+1}\parallel^2-\parallel\sqrt{a}\Lambda^\alpha_{h_x} u^{k-1}\parallel^2\right)\\
    I_2
  =&-\frac{(1-2\theta)}{4}\Big(\parallel\sqrt{a}(\Lambda^\alpha _{h_x} u ^{k+1}+\Lambda^\alpha _{h_x} u ^{k})\parallel^2-\parallel\sqrt{a}(\Lambda^\alpha _{h_x} u ^{k+1}-\Lambda^\alpha _{h_x} u ^{k})\parallel^2\\
  &~-\parallel\sqrt{a}(\Lambda^\alpha _{h_x} u ^k+\Lambda^\alpha _{h_x} u ^{k-1})\parallel^2+\parallel\sqrt{a}(\Lambda^\alpha _{h_x} u ^k-\Lambda^\alpha _{h_x} u ^{k-1})\parallel^2\Big),
\end{split}
\end{equation*}
and
   \begin{equation*}
   \begin{split}
  I_3
  =&-\theta\left(\parallel\sqrt{b}\Lambda_{h_y}^\beta  u^{k+1}\parallel^2-\parallel\sqrt{b}\Lambda_{h_y}^\beta  u^{k-1}\parallel^2\right),\\
  I_4
  =&-\frac{(1-2\theta)}{4}\Big(\parallel\sqrt{b}(\Lambda^\beta _{h_y} u^{k+1}+\Lambda^\beta _{h_y} u^{k})\parallel^2-\parallel\sqrt{b}(\Lambda^\beta _{h_y} u^{k+1}-\Lambda^\beta _{h_y} u^{k})\parallel^2\\
  &~-\parallel\sqrt{b}(\Lambda^\beta _{h_y} u ^{k}+\Lambda^\beta _{h_y} u ^{k-1})\parallel^2+\parallel\sqrt{b}(\Lambda^\beta _{h_y} u ^k-\Lambda^\beta _{h_y} u ^{k-1})\parallel^2\Big).
   \end{split}
   \end{equation*}
From  (\ref{3.14}) and (\ref{3.15}), we obtain
\begin{equation*}
\begin{split}
     &{\parallel  u^{k+1}_{\bar{t}}\parallel}^2- {\parallel  u^{k}_{\bar{t}}\parallel}^2
     +\theta^2\tau^6\parallel\sqrt{ab}\Lambda_{h_x}^\alpha\Lambda_{h_y}^\beta u^{k+1}_{\overline{t}}\parallel^2-\theta^2\tau^6\parallel\sqrt{ab}\Lambda_{h_x}^\alpha\Lambda_{h_y}^\beta u^{k}_{\overline{t}})\parallel^2
    -I_1-I_2-I_3-I_4\\
    &=(f^k, u^{k+1}- u^{k-1}),
\end{split}
\end{equation*}
i.e.,
  \begin{equation*}
  \begin{split}
  &{\parallel  u^{k+1}_{\bar{t}}\parallel}^2+\theta^2\tau^6\parallel\sqrt{ab}\Lambda_{h_x}^\alpha\Lambda_{h_y}^\beta u^{k+1}_{\overline{t}}\parallel^2+\theta\parallel\sqrt{a}\Lambda_{h_x}^\alpha  u^{k+1}\parallel^2\\
  &\quad+\frac{1-2\theta}{4}(\parallel\sqrt{a}(\Lambda_{h_x}^\alpha  u^{k+1}+\Lambda_{h_x}^\alpha  u^{k})\parallel^2
  -\parallel\sqrt{a}(\Lambda_{h_x}^\alpha  u^{k+1}-\Lambda_{h_x}^\alpha  u^{k})\parallel^2)\\
  &\quad+\theta\parallel\sqrt{b}\Lambda_{h_y}^\beta  u^{k+1}\parallel^2
  +\frac{1-2\theta}{4}(\parallel\sqrt{b}(\Lambda_{h_y}^\beta  u^{k+1}+\Lambda_{h_y}^\beta  u^{k})\parallel^2
   -\parallel\sqrt{b}(\Lambda_{h_y}^\beta  u^{k+1}-\Lambda_{h_y}^\beta  u^{k})\parallel^2)\\
  &={\parallel  u^{k}_{\bar{t}}\parallel}^2+\theta^2\tau^6\parallel\sqrt{ab}\Lambda_{h_x}^\alpha\Lambda_{h_y}^\beta u^{k}_{\overline{t}}\parallel^2+\theta\parallel\sqrt{a}\Lambda_{h_x}^\alpha  u^{k-1}\parallel^2\\
  &\quad+\frac{1-2\theta}{4}(\parallel\sqrt{a}(\Lambda_{h_x}^\alpha  u^{k}+\Lambda_{h_x}^\alpha  u^{k-1})\parallel^2
  -\parallel\sqrt{a}(\Lambda_{h_x}^\alpha  u^{k}-\Lambda_{h_x}^\alpha  u^{k-1})\parallel^2)\\
  &\quad+\theta\parallel\sqrt{b}\Lambda_{h_y}^\beta  u^{k}\parallel^2
  +\frac{1-2\theta}{4}(\parallel\sqrt{b}(\Lambda_{h_y}^\beta  u^{k}+\Lambda_{h_y}^\beta  u^{k-1})\parallel^2
  -\parallel\sqrt{b}(\Lambda_{h_y}^\beta  u^{k}-\Lambda_{h_y}^\beta  u^{k-1})\parallel^2)\\
  &\quad+(f^k, u^{k+1}- u^{k-1}).
   \end{split}
   \end{equation*}
Adding $\theta\parallel\sqrt{a}\Lambda_{h_x}^\alpha  u ^{k}\parallel^2+\theta\parallel\sqrt{b}\Lambda_{h_y}^\beta  u ^{k}\parallel^2$ on both sides of the above equation, we have
\begin{equation*}
\begin{split}
  &{\parallel  u ^{k+1}_{\bar{t}}\parallel}^2+\theta(\parallel\sqrt{a}\Lambda_{h_x}^\alpha  u ^{k+1}\parallel^2+\parallel\sqrt{a}\Lambda_{h_x}^\alpha u ^{k}\parallel^2)
  +\frac{1-2\theta}{4}(\parallel\sqrt{a}(\Lambda_{h_x}^\alpha  u^{k+1}+\Lambda_{h_x}^\alpha  u^{k})\parallel^2\\
  &\quad-\parallel\sqrt{a}(\Lambda_{h_x}^\alpha  u^{k+1}-\Lambda_{h_x}^\alpha  u^{k})\parallel^2)
  +\theta(\parallel\sqrt{b}\Lambda_{h_y}^\beta  u^{k+1}\parallel^2+\parallel\sqrt{b}\Lambda_{h_y}^\beta  u^{k}\parallel^2
  \\&\quad+\frac{1-2\theta}{4}(\parallel\sqrt{b}(\Lambda_{h_y}^\beta  u^{k+1}+\Lambda_{h_y}^\beta  u^{k})\parallel^2
  -\parallel\sqrt{b}(\Lambda_{h_y}^\beta  u^{k+1}-\Lambda_{h_y}^\beta  u^{k})\parallel^2)
  +\theta^2\tau^6\parallel\sqrt{ab}\Lambda_{h_x}^\alpha\Lambda_{h_y}^\beta u^{k+1}_{\overline{t}}\parallel^2\\
  &={\parallel  u ^{k}_{\bar{t}}\parallel}^2+\theta(\parallel\sqrt{a}\Lambda_{h_x}^\alpha  u ^{k}\parallel^2+\parallel\sqrt{a}\Lambda_{h_x}^\alpha u ^{k-1}\parallel^2)
  +\frac{1-2\theta}{4}(\parallel\sqrt{a}(\Lambda_{h_x}^\alpha  u^{k}+\Lambda_{h_x}^\alpha  u^{k-1})\parallel^2
  \\&\quad-\parallel\sqrt{a}(\Lambda_{h_x}^\alpha  u^{k}-\Lambda_{h_x}^\alpha  u^{k-1})\parallel^2)
  +\theta(\parallel\sqrt{b}\Lambda_{h_y}^\beta  u^{k}\parallel^2+\parallel\sqrt{b}\Lambda_{h_y}^\beta  u^{k-1}\parallel^2)
  \\&\quad+\frac{1-2\theta}{4}(\parallel\sqrt{b}(\Lambda_{h_y}^\beta  u^{k}+\Lambda_{h_y}^\beta  u^{k-1})\parallel^2
  -\parallel\sqrt{b}(\Lambda_{h_y}^\beta  u^{k}-\Lambda_{h_y}^\beta  u^{k-1})\parallel^2)\\
  &\quad+\theta^2\tau^6\parallel\sqrt{ab}\Lambda_{h_x}^\alpha\Lambda_{h_y}^\beta u^{k}_{\overline{t}}\parallel^2
 +(f^k, u^{k+1}- u^{k-1}).
\end{split}
\end{equation*}
Denoting
\begin{equation}\label{3.16}
\begin{split}
   E^k
&= {\parallel  u ^{k+1}_{\bar{t}}\parallel}^2+\theta(\parallel\sqrt{a}\Lambda_{h_x}^\alpha  u ^{k+1}\parallel^2+\parallel\sqrt{a}\Lambda_{h_x}^\alpha u ^{k}\parallel^2)
  +\frac{1-2\theta}{4}(\parallel\sqrt{a}(\Lambda_{h_x}^\alpha  u^{k+1}+\Lambda_{h_x}^\alpha  u^{k})\parallel^2\\
  &\quad-\parallel\sqrt{a}(\Lambda_{h_x}^\alpha  u^{k+1}-\Lambda_{h_x}^\alpha  u^{k})\parallel^2)
  +\theta(\parallel\sqrt{b}\Lambda_{h_y}^\beta  u^{k+1}\parallel^2+\parallel\sqrt{b}\Lambda_{h_y}^\beta  u^{k}\parallel^2\\
  &\quad+\frac{1-2\theta}{4}(\parallel\sqrt{b}(\Lambda_{h_y}^\beta  u^{k+1}+\Lambda_{h_y}^\beta  u^{k})\parallel^2
  -\parallel\sqrt{b}(\Lambda_{h_y}^\beta  u^{k+1}-\Lambda_{h_y}^\beta  u^{k})\parallel^2)\\
  &\quad+\theta^2\tau^6\parallel\sqrt{ab}\Lambda_{h_x}^\alpha\Lambda_{h_y}^\beta u^{k+1}_{\overline{t}}\parallel^2,
\end{split}
\end{equation}
 we have
   \begin{equation}\label{3.17}
     E^k = E^{k-1}+(f^k, u^{k+1}- u^{k-1}).
   \end{equation}
We rewrite (\ref{3.16}) as the following form 
\begin{equation}\label{3.18}
\begin{split}
   E^k
   &={\parallel u^{k+1}_{\bar{t}}\parallel}^2+\frac{1}{4}\parallel\sqrt{a}(\Lambda_{h_x}^\alpha  u^{k+1}+\Lambda_{h_x}^\alpha  u^{k})
   \parallel^2+\frac{1}{4}(4\theta-1)\parallel\sqrt{a}(\Lambda_{h_x}^\alpha  u^{k+1}-\Lambda_{h_x}^\alpha  u^{k})\parallel^2\\
   &\quad+\frac{1}{4}\parallel\sqrt{b}(\Lambda_{h_y}^\beta  u^{k+1}+\Lambda_{h_y}^\beta  u^{k})
   \parallel^2+\frac{1}{4}(4\theta-1)\parallel\sqrt{b}(\Lambda_{h_y}^\beta  u^{k+1}-\Lambda_{h_y}^\beta  u^{k})\parallel^2\\
   &\quad+\theta^2\tau^6\parallel\sqrt{ab}\Lambda_{h_x}^\alpha\Lambda_{h_y}^\beta u^{k+1}_{\overline{t}}\parallel^2,
\end{split}
\end{equation}
   where we use $$\parallel\sqrt{a}\Lambda_{h_x}^\alpha  u^{k+1}\parallel^2
   +\parallel\sqrt{a}\Lambda_{h_x}^\alpha  u^{k}\parallel^2=\frac{1}{2}(\parallel\sqrt{a}(\Lambda_{h_x}^\alpha  u^{k+1}+\Lambda_{h_x}^\alpha  u^{k})
   \parallel^2+\parallel\sqrt{a}(\Lambda_{h_x}^\alpha  u ^{k+1}-\Lambda_{h_x}^\alpha  u ^{k})\parallel^2),$$ and $$\parallel\sqrt{b}\Lambda_{h_y}^\beta u ^{k+1}\parallel^2
   +\parallel\sqrt{b}\Lambda_{h_y}^\beta  u^{k}\parallel^2=\frac{1}{2}(\parallel\sqrt{b}(\Lambda_{h_y}^\beta  u^{k+1}+\Lambda_{h_y}^\beta  u^{k})
   \parallel^2+\parallel\sqrt{b}(\Lambda_{h_y}^\beta  u^{k+1}-\Lambda_{h_y}^\beta  u^{k})\parallel^2).$$
According to
\begin{equation*}
\begin{split}
  (f^{k}, u^{k+1}- u^{k-1})
  &=2h_xh_y\tau\sum_{i=1}^{N_x-1}\sum_{j=1}^{N_y-1}f_{i,j}^k\left(\frac{ u_{i,j}^{k+1}- u_{i,j}^{k-1}}{2\tau}\right)\\
  &\leq h_xh_y\tau\sum_{i=1}^{N_x-1}\sum_{j=1}^{N_y-1}\left[\left(f_{i,j}^k\right)^2+\left(\frac{ u_{i,j}^{k+1}- u_{i,j}^{k}+ u_{i,j}^{k}- u_{i,j}^{k-1}}{2\tau}\right)^2\right]\\
  &\leq \frac{\tau}{2}\left({\parallel  u^{k+1}_{\bar{t}}\parallel}^2 +{\parallel  u^{k}_{\bar{t}}\parallel}^2\right)+\tau||f^k||^2,
\end{split}
\end{equation*}
and  (\ref{3.18}),  (\ref{3.17}),  there exists
\begin{equation*}
\begin{split}
E^{k}-E^{k-1}=(f^{k},u^{k+1}-u^{k-1}) \leq \frac{\tau}{2} (E^{k}+E^{k-1}) +\tau||f^k||^2,
\end{split}
\end{equation*}
i.e.,
\begin{equation*}
\begin{split}
\left(1-\frac{\tau}{2}\right)E^{k}\leq \left(1+\frac{\tau}{2}\right)E^{k-1}+\tau||f^k||^2.
\end{split}
\end{equation*}
For $\tau\leq 2/3$, which leads to 
\begin{equation*}
\begin{split}
E^{k}\leq \left(1+\frac{3\tau}{2}\right)E^{k-1}+\frac{3}{2}\tau||f^k||^2.
\end{split}
\end{equation*}
From Lemma \ref{lemma3.2},  there exists 
\begin{equation*}
\begin{split}
E^k \leq e^{\frac{3}{2}k\tau}\left[  E^0+\frac{3}{2}\tau\sum_{l=1}^k||f^l||^2\right].
\end{split}
\end{equation*}
The proof is completed.
\end{proof}
\begin{theorem}\label{theorem3.3}
 Let $u(x_i,y_j,t_k)$ be the exact solution of (\ref{1.1}) with $1<\alpha,\beta\leq2$, $u^k_{i,j}$ be the solution of (\ref{2.22}) and $e_{ij}^k=u(x_i,y_j,t_k)-u_{ij}^k$. Then
\begin{equation*}
\begin{split}
E^k=\mathcal{O}(\tau^2+h_x^2+h_y^2)^2,
\end{split}
\end{equation*}
where the energy norm is defined by 
\begin{equation*}
\begin{split}
   E^k
   &={\parallel e^{k+1}_{\bar{t}}\parallel}^2+\frac{1}{4}\parallel\sqrt{a}(\Lambda_{h_x}^\alpha  e^{k+1}+\Lambda_{h_x}^\alpha  e^{k})
   \parallel^2+\frac{1}{4}(4\theta-1)\parallel\sqrt{a}(\Lambda_{h_x}^\alpha  e^{k+1}-\Lambda_{h_x}^\alpha  e^{k})\parallel^2\\
   &\quad+\frac{1}{4}\parallel\sqrt{b}(\Lambda_{h_y}^\beta  e^{k+1}+\Lambda_{h_y}^\beta  e^{k})
   \parallel^2+\frac{1}{4}(4\theta-1)\parallel\sqrt{b}(\Lambda_{h_y}^\beta  e^{k+1}-\Lambda_{h_y}^\beta  e^{k})\parallel^2\\
   &\quad+\theta^2\tau^6\parallel\sqrt{ab}\Lambda_{h_x}^\alpha\Lambda_{h_y}^\beta e^{k+1}_{\overline{t}}\parallel^2.
\end{split}
\end{equation*}
\end{theorem}
\begin{proof}
Subtracting  (\ref{2.22}) from  (\ref{2.23}),  it yields
\begin{equation*}
\begin{split}
   &\frac{1}{\tau^2}\delta^2_te^k_{i,j}+\theta^2 \tau^4 a_{i,j}b_{i,j}\nabla^\alpha_{h_x}\nabla^\beta_{h_y} \big(e_{i,j}^{k+1}-2e_{i,j}^{k}+e_{i,j}^{k-1}\big)\\
   &=\theta a_{i,j}\nabla^\alpha_{h_x}e^{k+1}_{i,j}+(1-2\theta)a_{i,j}\nabla^\alpha_{h_x}e^k_{i,j}+\theta a_{i,j}\nabla^\alpha_{h_x}e^{k-1}_{i,j}\\
   &\quad+\theta b_{i,j}\nabla_{h_y}^\beta e^{k+1}_{i,j}+(1-2\theta)b_{i,j}\nabla_{h_y}^\beta e^k_{i,j}+\theta b_{i,j}\nabla_{h_y}^\beta e^{k-1}_{i,j}+\widetilde{R}^{k}_{i,j}.
\end{split}
\end{equation*}
Using Lemma \ref{lemma3.8}, there exists
\begin{equation}\label{3.19}
\begin{split}
E^k \leq e^{\frac{3}{2}k\tau}\left[  E^0+\frac{3}{2}\tau\sum_{l=1}^k||\widetilde{R}^l||^2\right]
\end{split}
\end{equation}
with the energy norm 
\begin{equation}\label{3.20}
\begin{split}
 E^k
   &={\parallel e^{k+1}_{\bar{t}}\parallel}^2+\frac{1}{4}\parallel\sqrt{a}(\Lambda_{h_x}^\alpha  e^{k+1}+\Lambda_{h_x}^\alpha  e^{k})
   \parallel^2+\frac{1}{4}(4\theta-1)\parallel\sqrt{a}(\Lambda_{h_x}^\alpha  e^{k+1}-\Lambda_{h_x}^\alpha  e^{k})\parallel^2\\
   &\quad+\frac{1}{4}\parallel\sqrt{b}(\Lambda_{h_y}^\beta  e^{k+1}+\Lambda_{h_y}^\beta  e^{k})
   \parallel^2+\frac{1}{4}(4\theta-1)\parallel\sqrt{b}(\Lambda_{h_y}^\beta  e^{k+1}-\Lambda_{h_y}^\beta  e^{k})\parallel^2\\
   &\quad+\theta^2\tau^6\parallel\sqrt{ab}\Lambda_{h_x}^\alpha\Lambda_{h_y}^\beta e^{k+1}_{\overline{t}}\parallel^2.
\end{split}
\end{equation}
Next we estimate the local error truncation of $E^0$. Since $e_{i,j}^0=0$,
~~$\theta^2\tau^6\parallel\sqrt{ab}\Lambda_{h_x}^\alpha\Lambda_{h_y}^\beta e^{1}_{\overline{t}}\parallel^2=\mathcal{O}(\tau^6)$ in (\ref{3.14}) and
\begin{equation*}
\begin{split}
    e_{i,j}^1
 &=\frac{\tau^2}{2} \Big[a(x_i,y_j)\left(\frac{\partial^\alpha u(x_i,y_j,0)}{\partial |x|^\alpha}-\nabla^\alpha_{h_x}\varphi(x_i,y_j)\right)\\
 &\quad+b(x_i,y_j)\Big(\frac{\partial^\beta u(x_i,y_j,0)}{\partial |y|^\beta}-\nabla^\beta_{h_y}\varphi(x_i,y_j)
 \Big)\Big]
+\frac{1}{2}\int_0^\tau(\tau-t)^2\frac{\partial^3 u(x_i,y_j,t)}{\partial t^3}dt\\
 &=\frac{\tau^2}{2}\Big(a(x_i,y_j)C_{1,\alpha}\frac{\partial^{\alpha+2}u(\xi_i,y_j,t)}{\partial |x|^{\alpha+2}}h_x^2+b(x_i,y_j)C_{1,\beta}\frac{\partial^{\beta+2}u(x_i,\eta_j,t)}{\partial |y|^{\beta+2}}h_y^2\Big)\\
 &\quad+\frac{1}{2}\int_0^\tau{(\tau-t)^2\frac{\partial^3u(x_i,y_j,t)}{\partial t^3}dt}
\leq C_{1,\alpha,\beta}(\tau^3+\tau^2h_x^2+\tau^2h_y^2).
\end{split}
\end{equation*}
Here  the coefficients  $ C_{1,\alpha}$ and $C_{1,\beta}$ are the  constants independent of $h$, $\tau$ and 
\begin{equation*}
\begin{split}
C_{1,\alpha,\beta}
&= \max\limits_{0\leq x\leq x_r,0\leq y\leq y_r,0\leq t\leq T}\Big\{\frac{1}{2}a_1\Big|C_{1,\alpha}\frac{\partial^{\alpha+2}u(\xi_i,y_j,t)}{\partial |x|^{\alpha+2}}\Big|,\\
&\qquad\frac{1}{2}b_1\Big|C_{1,\beta}\frac{\partial^{\beta+2}u(x_i,\eta_j,t)}{\partial |y|^{\beta+2}}\Big|,\frac{1}{6}\Big|\int_0^\tau\frac{\partial^3u(x_i,y_j,t)}{\partial t^3}dt\Big|\Big\}.
\end{split}
\end{equation*}
Then we obtain 
\begin{equation}\label{3.21}
\begin{split}
   \parallel e^1_{\overline{t}}\parallel^2
   &=\parallel \frac{e^1-e^0}{\tau}\parallel^2\\
   &\leq(N_x-1)h_x(N_y-1)h_y\frac{1}{\tau^2}C_{1,\alpha,\beta}(\tau^3+\tau^2h_x^2+\tau^2h_y^2)\cdot C_{1,\alpha,\beta}(\tau^3+\tau^2h_x^2+\tau^2h_y^2)\\
   &\leq C_{1,\alpha,\beta}^2x_ry_r(\tau^2+\tau h_x^2+\tau h_y^2)^2.
\end{split}
\end{equation}
From (\ref{1.1}) and the above equations, there exists
\begin{equation*}
\begin{split}
&\parallel\sqrt{a}\Lambda_{h_x}^\alpha e^{1}\parallel^2
=-(a\nabla_{h_x}^\alpha e^{1},e^{1})=-h_xh_y\sum_{i=1}^{N_x-1}\sum_{j=1}^{N_y-1}a_{i,j}\left(\nabla_{h_x}^\alpha e^{1}_{i,j}\right)\cdot e_{i,j}^1\\
&=-h_xh_y\sum_{i=1}^{N_x-1}\sum_{j=1}^{N_y-1}a_{i,j} \sum_{l=0}^{N_x}\frac{-\kappa_{\alpha}}{h_x^\alpha}\varphi_{i,l}^{\alpha}
\Big[\frac{a_{i,j}C_{1,\alpha}}{2}\frac{\partial^{\alpha+2}u(\xi_i,y_j,t)}{\partial |x|^{\alpha+2}}\tau^2h_x^2\\
&\quad+\frac{b_{i,j}C_{1,\beta}}{2}\frac{\partial^{\beta+2}u(x_i,\eta_j,t)}{\partial |y|^{\beta+2}}\tau^2h_y^2
+\frac{1}{2}\int_0^\tau{(\tau-t)^2\frac{\partial^3u(x_i,y_j,t)}{\partial t^3}dt}\Big]\cdot e_{i,j}^1\\
&= -h_xh_y\sum_{i=1}^{N_x-1}\sum_{j=1}^{N_y-1}a_{i,j}
\Big[\frac{a_{i,j}C_{1,\alpha}}{2}\frac{\partial^{2\alpha+2}u(\xi_i,y_j,t)}{\partial |x|^{2\alpha+2}}\tau^2h_x^2
+C_{2,\alpha}\frac{\partial^{2\alpha+4}u(\widehat{\xi}_{i},y_{j},t)}{\partial |x|^{2\alpha+4}}\tau^2h_x^4\\
&\quad+\frac{b_{i,j}C_{1,\beta}}{2}\frac{\partial^{\alpha+\beta+2}u(x_i,\eta_j,t)}{\partial |x|^{\alpha}|y|^{\beta+2}}\tau^2h_y^2
+C_{2,\alpha,\beta}\frac{\partial^{\alpha+\beta+4}u(\overline{\xi_{i}},\eta_{j},t)}{\partial |x|^{\alpha+2}|y|^{\beta+2}}\tau^2h_x^2h_y^2\Big]\cdot e_{i,j}^1\\
&\quad-h_xh_y\sum_{i=1}^{N_x-1}\sum_{j=1}^{N_y-1}\frac{a_{i,j}}{2}\int_0^\tau{(\tau-t)^2
\left[\frac{\partial^{\alpha+3}u(x_i,y_j,t)}{\partial t^3\partial |x|^\alpha}
+C_{3,\alpha}\frac{\partial^{\alpha+5}u(\widetilde{\xi}_{i},y_{j},t)}{\partial t^3 \partial |x|^{\alpha+2}}h_x^2\right]dt}\cdot e_{i,j}^1\\
&   \leq  C_{3,\alpha,\beta}(\tau^3+\tau^2h_x^2+\tau^2h_y^2)\cdot C_{1,\alpha,\beta}(\tau^3+\tau^2h_x^2+\tau^2h_y^2),
\end{split}
\end{equation*}
where $\xi_{i},\widehat{\xi}_{i},\overline{\xi_{i}}, \widetilde{\xi_{i}} \in (0,x_r)$ and $ C_{l,\alpha}$ and $C_{l,\alpha,\beta}, 1\leq l\leq 3$ are the  constants.
Similarly, we have
\begin{equation*}
\begin{split}
&\parallel\sqrt{b}\Lambda_{h_y}^\beta e^{1}\parallel^2\leq  \widetilde{C}_{3,\alpha,\beta}(\tau^3+\tau^2h_x^2+\tau^2h_y^2)\cdot \widetilde{C}_{1,\alpha,\beta}(\tau^3+\tau^2h_x^2+\tau^2h_y^2)
\end{split}
\end{equation*}
with the constants $\widetilde{C}_{1,\alpha,\beta}$ and $\widetilde{C}_{3,\alpha,\beta} $.

According to  (\ref{3.20}), (\ref{3.21}) and the above equations, we get
\begin{equation}\label{3.22}
 E^0\leq  C_{\alpha,\beta}^2x_ry_r(\tau^2+\tau h_x^2+\tau h_y^2)^2,
\end{equation}
where   $C_{\alpha,\beta}$ is a constant.
Hence, using (\ref{2.17}), (\ref{3.19}) and  (\ref{3.22}), there exists
\begin{equation*}
\begin{split}
E^k
&\leq e^{\frac{3}{2}k\tau}\left[  C_{\alpha,\beta}^2x_ry_r(\tau^2+\tau h_x^2+\tau h_y^2)^2+\frac{3}{2}k\tau C^2_{u,\alpha,\beta}(\tau^2+h_x^2+h_y^2)^2\right]\\
&\leq  \widetilde{C}_{\alpha,\beta} e^{\frac{3}{2}T}(\tau^2+h_x^2+h_y^2)^2
\end{split}
\end{equation*}
with $ \widetilde{C}_{\alpha,\beta}=2\max\{ C_{\alpha,\beta}^2x_ry_r,\frac{3}{2}C^2_{u,\alpha,\beta}T\}$.
The proof is completed.
\end{proof}
\begin{theorem}\label{theorem3.4}
  The difference scheme (\ref{2.27}) with $1<\alpha,\beta\leq2$ and $\frac{1}{4}\leq\theta\leq1$ is unconditionally stable.
\end{theorem}
\begin{proof}
  From Lemma \ref{lemma3.7}, the result  is obtained.
\end{proof}
\begin{remark}
The operator $\mathcal{L}$ appears in the nonlocal wave equation \cite{Du:12}
\begin{equation*}
\left\{ \begin{split}
\frac{\partial^2u(x,t)}{\partial t^2} - \mathcal{L}_\delta u(x,t)   &=f_\delta(x,t) &  ~~{\rm on}  & ~~\Omega,\, t>0,\\
                        u(x,0)&=u_0      &  ~~{\rm on}  & ~~\Omega \cup \Omega_\mathcal{I},\\
                 u&=g        &  ~~{\rm on}  & ~~ \Omega_\mathcal{I},t>0.
 \end{split}
 \right.
\end{equation*}
From \cite{CD:16}, we known that the approximation operator of $ -\mathcal{L}_\delta$ is also the symmetric positive definite.
Hence, the framework  of the  stability  and  convergence analysis are still valid for  the nonlocal wave equation.
\end{remark}
\section{Numerical results}
In this section, we numerically verify the above theoretical results and the $ l_\infty$ norm is used to measure the numerical errors.

\begin{example}Consider  the space-Riesz fractional wave equation (\ref{2.5}),
on a finite domain  $0< x < 1 $,  $0<t \leq 1$  with the coefficient $d(x)=x^{\alpha}$, the forcing function is
\begin{equation*}
\begin{split}
f(x,t)=
&e^{-t}x^2(1-x)^2\\
&+\frac{x^{\alpha}e^{-t}}{2\cos(\alpha \pi/2)}
\left[
   \Gamma(5)\frac{x^{4-\alpha}+(1-x)^{4-\alpha}}{\Gamma(5-\alpha)}-2\Gamma(4)\frac{x^{3-\alpha}+(1-x)^{3-\alpha}}{\Gamma(4-\alpha)}+\Gamma(3)\frac{x^{2-\alpha}+(1-x)^{2-\alpha}}{\Gamma(3-\alpha)}   \right]
\end{split}
\end{equation*}
with the initial conditions $u(x,0)=x^2(1-x)^2$, $\frac{\partial}{\partial t}u(x,0)=-x^2(1-x)^2$, and the boundary conditions $u(0,t)=u(1,t)=0$.
The exact solution of the fractional PDEs  is $$u(x,t)=e^{-t}x^2(1-x)^2.$$
\begin{table}  []  \fontsize{9.5pt}{12pt}  \selectfont
\begin{center}
  \caption{The maximum errors  and convergent orders for (\ref{2.14})  with $\tau=h$.}\vspace{5pt}
\begin{tabular*}{\linewidth}{@{\extracolsep{\fill}}*{7}{c}}                        \hline
   $\tau$& $\alpha=1.3,\theta=0.25$ & Rate   & $\alpha=1.6,\theta=0.5$ & Rate     & $\alpha=1.9,\theta=1$ & Rate      \\\hline
~~1/40&            8.8516e-05&            &               9.0532e-05&          &             7.5678e-05& \\
~ 1/80&            2.2156e-05&      1.9983&               2.2329e-05&   2.0195 &             1.9487e-05&  1.9574 \\
~1/160&            5.5242e-06&      2.0038&               5.5161e-06&   2.0172 &             4.7477e-06&  2.0372  \\
~1/320&            1.3761e-06&      2.0052&               1.3636e-06&   2.0163 &             1.1581e-06&  2.0355 \\\hline
\end{tabular*}{\label{tab:1}}
\end{center}
\end{table}
Table \ref{tab:1} shows that  the scheme (\ref{2.14}) is second order convergent in both space and time directions.
\end{example}

\begin{example} Consider the two-dimensional space-Riesz fractional wave  equation (\ref{1.1}),
on a finite domain $ 0<  x< 1,\,0<  y< 1$, $0< t \leq 1/2$  with the variable coefficients
\begin{equation*}
\begin{split}
d(x,y)=x^{\alpha}y, \quad e(x,y)=xy^{\beta},
\end{split}
\end{equation*}
and the initial conditions $u(x,y,0)=\sin(1)x^2(1-x)^2y^2(1-y)^2$, $\frac{\partial}{\partial t}u(x,y,0)=\cos(1)x^2(1-x)^2y^2(1-y)^2$ with the
zero Dirichlet boundary conditions on the rectangle. The exact solution of the  PDEs  is
\begin{equation*}
\begin{split}
u(x,y,t)=\sin(t+1)x^2(1-x)^2y^2(1-y)^2.
\end{split}
\end{equation*}
Using  the above  conditions, it is easy to obtain the forcing function $f(x,y,t)$.
\begin{table}[]\fontsize{9.5pt}{12pt}\selectfont
\begin{center}
  \caption{The maximum errors  and convergent orders for (\ref{2.27})  with $\tau=h_x=h_y$ and $\theta=0.75$.}\vspace{5pt}
\begin{tabular*}{\linewidth}{@{\extracolsep{\fill}}*{7}{c}}                        \hline
   $\tau$& $\alpha=1.3,\beta=1.7$ & Rate   & $\alpha=1.5,\beta=1.5$ & Rate     & $\alpha=1.7,\beta=1.3$ & Rate      \\\hline
~~1/20&            1.4066e-04&            &               1.4066e-04&          &              1.4500e-04& \\
~ 1/40&            3.8290e-05&      1.8772&               3.7449e-05&   1.9093 &              3.7041e-05&   1.9688 \\
~~1/80&            9.4992e-06&      2.0111&               9.4992e-06&   1.9790 &              9.4992e-06&  1.9632  \\
~1/160&            2.4049e-06&      1.9818&               2.4049e-06&   1.9818 &              2.4049e-06&  1.9818 \\\hline
\end{tabular*}\label{tab:2}
\end{center}
\end{table}
Table \ref{tab:2} shows that  the scheme (\ref{2.27}) is second order convergent in both space and time directions.
\end{example}

\section{Conclusion}
In this work we have developed the energy method to estimate the two-dimensional  space-Riesz fractional wave  equation with the variable coefficients.
To the best of our knowledge,   the  convergence and stability are lack of study  for the one-dimensional space-Riesz fractional wave equation with the nonzero conditions.
In this paper, the priori error estimates have been  established and the convergence analysis and stability  of the proposed method have been proved.
For two-dimensional cases with the variable coefficients, the discretized matrices  are proved to be commutative, which ensures to carry out of the priori error estimates.
Numerical results have been given to illustrate the robustness  and efficiency of the presented method with the second order convergence.
We remark that though this current paper focus on the space-Riesz fractional wave  equation, the energy estimates is still valid for the compact finite difference schemes and the nonlocal wave equation \cite{Du:12}.

\section*{Acknowledgments}This work was supported by NSFC 11601206, the Fundamental Research Funds for the Central Universities under Grant No. lzujbky-2016-105, and SIETP 201710730065.

\end{document}